\documentclass[a4paper,12pt,english]{smfart}

\usepackage{amssymb}
\usepackage{url}
\usepackage[T1]{fontenc}
\RequirePackage{calrsfs}
\DeclareSymbolFont{rsfscript}{OMS}{rsfs}{m}{b}
\DeclareSymbolFontAlphabet{\mathrsfs}{rsfscript}
\usepackage[utopia,expert]{mathdesign}
\usepackage{lscape}

\usepackage{palatino}
\usepackage{rotating}
\usepackage{graphicx}

\usepackage{enumerate}

\usepackage{color}
\definecolor{shadecolor}{gray}{0.90}

\def\bfit{\bfseries\itshape}

\def\proofname{D\'emonstration}

\input xypic
\xyoption{all}
\xyoption{arc}

\newtheorem{theo}{Theorem}[section]
\def\thetheo{\thesection.\arabic{theo}}
\newtheorem{prop}[theo]{Proposition}

\newtheorem{lem}[theo]{Lemma}

\newtheorem{coro}[theo]{Corollary}

\renewcommand\theequation{\thetheo}
\def\equat{\refstepcounter{theo}\begin{equation}}
\def\endequat{\end{equation}}

\renewcommand\thetable{\thesection.\arabic{table}}

\renewcommand\thesection{\arabic{section}}
\renewcommand\thesubsection{\thesection.\Alph{subsection}}

\def\contentsname{Table des mati\`eres}
\def\listfigurename{Liste des figures}
\def\listtablename{Liste des tables}
\def\appendixname{Appendice}

\def\refname{R\'ef\'erences}

\def\CG{{\mathfrak C}}  \def\cG{{\mathfrak c}}

    \def\NM{{\mathbb{N}}}

    \def\RM{{\mathbb{R}}}
\def\SG{{\mathfrak S}}

    \def\ZM{{\mathbb{Z}}}

  \def\ab{{\mathbf a}}  
  \def\bb{{\mathbf b}}  
    \def\CC{{\mathcal{C}}}
    \def\DC{{\mathcal{D}}}

    \def\GC{{\mathcal{G}}}
    \def\HC{{\mathcal{H}}}
\def\Ib{{\mathbf I}}

    \def\MC{{\mathcal{M}}}

    \def\PC{{\mathcal{P}}}

    \def\SC{{\mathcal{S}}}

\def\Arm{{\mathrm{A}}}

\def\Jrm{{\mathrm{J}}}

\def\Zrm{{\mathrm{Z}}}

          \def\aov{{\overline{a}}}
          
          \def\cov{{\overline{c}}}

\def\a{\alpha}

\def\g{\gamma}
\def\G{\Gamma}
\def\d{\delta}

\def\ph{\varphi}

\def\L{\Lambda}

\def\o{\omega}
\def\O{\Omega}

\DeclareMathOperator{\Ind}{{\mathrm{Ind}}}

\DeclareMathOperator{\Irr}{{\mathrm{Irr}}}

\DeclareMathOperator{\Res}{{\mathrm{Res}}}

\def\to{\rightarrow}

\def\DS{\displaystyle}

\def\SSS{\scriptscriptstyle}

\def\lexp#1#2{\kern\scriptspace\vphantom{#2}^{#1}\kern-\scriptspace#2}
\def\le{\hspace{0.1em}\mathop{\leqslant}\nolimits\hspace{0.1em}}

\mathchardef\inferieur="321E
\mathchardef\superieur="321F

\def\eqna{\begin{eqnarray*}}
\def\endeqna{\end{eqnarray*}}

\def\itemth#1{\item[${\mathrm{(#1)}}$]}

\catcode`\@=11
\long\def\@car#1#2\@nil{#1}
\long\def\@first#1#2{#1}
\long\def\@second#1#2{#2}
\long\def\ifempty#1{\expandafter\ifx\@car#1@\@nil @\@empty
  \expandafter\@first\else\expandafter\@second\fi}
\catcode`\@=12

\renewcommand{\tocsection}[3]{%
  {\bf \indentlabel{\ifempty{#2}{}{\ignorespaces#1 #2.\quad}}#3}}

\theoremstyle{remark}
\newtheorem{rema}[theo]{Remark}
\newtheorem{exemple}[theo]{Example}
\newtheorem{defi}[theo]{Definition}

\theoremstyle{plain}

\newtheorem{conjecture}[theo]{Conjecture}

\def\xyinj{\ar@{^{(}->}}
\def\xysur{\ar@{->>}}

\def\isomorphisme#1{[#1]}

\bigskip

\def\petitespace{\vphantom{$\DS{\frac{\DS{A^A}}{\DS{A_A}}}$}}

\makeatletter
\def\hlinewd#1{%
\noalign{\ifnum0=`}\fi\hrule \@height #1 %
\futurelet\reserved@a\@xhline}
\makeatother

\newlength\epaisLigne

\usepackage{multirow}

\newcommand{\longiso}{\stackrel{\sim}{\longrightarrow}}

\def\bruhatordre{\hspace{0.1em}\mathop{{\boldsymbol{\leqslant}}}\nolimits\hspace{0.1em}}
\def\bruhatordrestrict{\hspace{0.1em}{\boldsymbol{<}}\hspace{0.1em}}

\def\egal{{\SSS{=}}}

\def\prel{\leqslant_{L}}

\def\prer{\leqslant_{R}}
\def\prelr{\leqslant_{LR}}

\def\siml{\sim_{L}}
\def\simr{\sim_{R}}
\def\simlr{\sim_{LR}}
\DeclareMathOperator{\Cell}{{\mathrm{Cell}}}
\def\isomorphisme#1{[#1]}
\def\Tgras#1{{\boldsymbol{T}}_{\!\! #1}}

\def\springer#1{\begin{centerline}{\fcolorbox{black}{shadecolor}{~
    \begin{minipage}[t]{0.8\textwidth}{\vphantom{~}{\itshape #1}\vphantom{$A_{\DS{A_A}}$}}
            \end{minipage}~}}\end{centerline}\medskip}

\begin{document}

\baselineskip=16pt
%\large\baselineskip=20pt
%\Large\baselineskip=24pt

\title{Conjugacy classes of involutions and Kazhdan--Lusztig cells}

\author{{\sc C\'edric Bonnaf\'e}}
\address{
Institut de Math\'ematiques et de Mod\'elisation de Montpellier 
(CNRS: UMR 5149), 
Universit\'e Montpellier 2,
Case Courrier 051,
Place Eug\`ene Bataillon,
34095 MONTPELLIER Cedex,
FRANCE} 

\makeatletter
\email{cedric.bonnafe@math.univ-montp2.fr}
\makeatother

\author{{\sc Meinolf Geck}}

\address{Institute of Mathematics, 
Aberdeen University, 
ABERDEEN AB24 3UE, 
SCOTLAND, UK}
\email{m.geck@abdn.ac.uk}

%\makeatother

%\subjclass{According to the 2000 classification:
%Primary ???; Secondary ???}

\date{\today}

% \thanks{The author is partly 
% supported by the ANR (Project No JC07-192339)}

\begin{abstract} 
According to an old result of Sch\"utzenberger, the involutions in 
a given two-sided cell of the symmetric group $\SG_n$ are all conjugate. 
In this paper, we study possible generalisations of this property to 
other types of Coxeter groups. We show that Sch\"utzenberger's result 
is a special case of a general result on ``smooth'' two-sided cells.
Furthermore, we consider Kottwitz' conjecture concerning the intersections 
of conjugacy classes of involutions with the left cells in a finite 
Coxeter group. Our methods lead to a proof of this conjecture for 
classical types; combined with previous work, this leaves type $E_8$ 
as the only remaining open case.
\end{abstract}

\maketitle

\pagestyle{myheadings}

\markboth{\sc C. Bonnaf\'e and M. Geck}{\sc Involutions and cells}

% \tableofcontents
% 
% \vskip1cm

\def\marginparwidth{2.6cm}

Let $(W,S)$ be a Coxeter system, and let $\ph : S \to \Gamma_{>0}$ 
be a {\it weight function}, that is, a map with values in a totally
ordered abelian group $\Gamma$ such that $\ph(s)=\ph(t)$ whenever 
$s$ and $t$ are conjugate in $W$. Associated with this datum, G. Lusztig has 
defined~\cite{l} a partition of $W$ into left, right or two-sided cells.
(If $\ph$ is constant, then this was defined earlier by D. Kazhdan and 
G. Lusztig~\cite{KL}). There seems to be almost no connection between cells 
and conjugacy classes of elements of order greater than $2$. However,
several papers have investigated links between conjugacy classes of 
involutions and cells~\cite{kottwitz},~\cite{luvo},~\cite{luvo1},
\cite{marberg}.

To the best of our knowledge, the oldest result in this direction is the 
following. In the case where $W=\SG_n$, the two-sided cells are described 
by the Robinson-Schensted correspondence~\cite{KL}, \cite{Lu2}. It then 
follows from a result of M.-P. Sch\"utzenberger~\cite{sch} (see also 
C. Hohlweg~\cite{hohlweg}) that, if $W=\SG_n$, then all the (Duflo) 
involutions contained in the same two-sided cell are conjugate.  Of course, 
as it can be seen already in the Coxeter group of type $B_2$ (with $\ph$ 
constant), the same kind of result cannot be generalized as such. However, 
again if $W$ is of type $B_2$ but if we now take $\ph$ to be non-constant,
then again the same result holds. In this paper we shall investigate
possible generalizations of M.-P. Sch\"utzenberger's result. For any 
subset $X\subseteq W$, we denote by $\CC_2(X)$ the union of all conjugacy
classes of involutions in $W$ which have non-empty intersection with $X$.

\medskip

\begin{quotation}
{\normalsize
{\bfit Conjecture.} --- {\it If $C$ and $C'$ are two left cells contained in
the same two-sided cell, then $\CC_2(C)=\CC_2(C')$.}  }
\end{quotation}

\medskip

As it can already be checked in type $A_3$, the obvious generalization 
of this conjecture to elements of any order is false. In this paper, we 
investigate this conjecture whenever $W$ is finite. For simplification, 
all along this paper, we will say that {\it ``Lusztig's Conjectures P's 
hold''} if {\it ``Lusztig's Conjectures P1, P2,\dots, P15 
in~\cite[Chapter 14]{lusztig} hold''}). Our aim is to prove the following 
result.

\medskip

\noindent{\bfit Theorem.} ---
{\it {\bfit Assume that $W$ is finite.} 
Let $\CG$ be a two-sided cell of $W$ and let $C$ and $C'$ be two left cells 
contained in $\CG$. Then:
\begin{itemize}
\itemth{a} If $\ph$ is constant, then $\CC_2(C)=\CC_2(C')$.

\itemth{b} {\bfit Assume that Lusztig's Conjectures P's for $(W,S,\ph)$ hold}. 
If $\CG$ is a {\bfit smooth} two-sided cell, then all the involutions in 
$\CG$ are conjugate. In particular, $\CC_2(C)=\CC_2(C')$ is a single 
conjugacy class.

\itemth{c} {\bfit Assume that $\ph$ is constant and $W$ is of classical
type $B_n$ or $D_n$}. Then $|\CC \cap C|=|\CC \cap C'|$ where $\CC$
is any conjugacy class of involutions in $W$.
%\itemth{c} If~\cite[Conjecture~5.11]{geck pycox} holds, then $\CC_2(C)=
%\CC_2(C')$.
\end{itemize}}

\medskip
Here, a two-sided cell $\CG$ is called {\it smooth} if the family of 
irreducible characters associated with $\CG$ contains only one element. 
This definition is inspired by the theory of rational Cherednik algebras and 
Calogero--Moser cells (as developed in~\cite{note} or~\cite{cmcells}). 
Note that smooth two-sided cells actually occur quite often; for example, all the 
two-sided cells of $\SG_n$ are smooth~\cite{KL}, \cite{Lu2}, as well as all 
the two-sided cells when $W$ is of type $B_n$ and $\ph$ corresponds 
to the asymptotic case as in~\cite{bonnafe iancu}, \cite{bonnafe two}. 
See also Table~\ref{table:eq} (p.~\pageref{table:eq}) for more numerical 
data.

Part (a) of the theorem will be shown in Proposition~\ref{eqpa1}; for
part (b) see Corollary~\ref{coro:conjugues}. An essential ingredient in 
our proof is the fact that, if $\CC$ is a conjugacy class of involutions 
in $W$, then $\sum_{w \in \CC} T_w$ is central in the Hecke algebra $\HC$; 
see Section~\ref{sec1}. (Here, $(T_w)_{w \in W}$ is the standard basis of 
$\HC$, as explained below.) Whenever $\CC$ is a conjugacy class of 
reflections, this result is due to L. Iancu (unpublished). 

Part (c) of the above result would follow from results of Lusztig
\cite[Chap.~12]{LuB} and a general conjecture due to R.E. Kottwitz 
\cite{kottwitz} concerning the intersections of conjugacy classes of 
involutions with left cells. Here, we prove (c) directly by the methods 
developped in Section~\ref{section:eqpa} and then use this to actually show
that Kottwitz' conjecture holds for $W$ of classical type; see 
Sections~\ref{section:kottBn}, \ref{section:kottDn} and \cite{tkott}.

Finally, we point out that our conjecture also makes sense for arbitrary
Coxeter groups. It can be checked easily that it holds in the infinite 
dihedral case; it also follows from work of J. Guilhot~\cite{guilhot} that, 
if $W$ is affine and if $\CG_0$ is the lowest two-sided cell, then the above conjecture 
holds for left cells contained in $\CG_0$. 
% but it would be interesting to check if it also holds in 
% the cases where the description of cells is known, for example, in the 
% affine type $\tilde{A}_2$, $\tilde{B}_2$ or $\tilde{G}_2$ for any $\ph$; 
% see~\cite{guilhot}. 

\section{Hecke algebras, involutions, cells} \label{sec1}

Let $(W,S)$ be a {\it finite} Coxeter system, let $\ell : W \to \NM$ denote the 
length function, let $\G$ be a totally ordered abelian group and let $\ph : 
S \to \G_{>0}$ be a {\it weight function} that is, a map such that 
$\ph(s)=\ph(t)$ whenever $s$ and $t$ are conjugate in $W$. We denote by 
$A$ the group ring $\RM[\G]$, denoted exponentially: in other words, 
$A=\oplus_{\g \in \G} \RM v^\g$, with $v^\g v^{\g'} = v^{\g+\g'}$. If 
$a=\sum_{\g \in \G} a_\g v^\g \in A$, then we denote by $\deg(a)$ its 
degree, namely the maximal $\g \in \G$ such that $\a_\g \neq 0$ (note that 
$\deg(0) = -\infty$). 

We denote by $\HC=\HC(W,S,\ph)$ the Hecke algebra with parameter $\ph$. 
As a module, $\HC=\oplus_{w \in W} ~A~\! T_w$ and the multiplication is 
completely determined by the following two rules:
\[ \begin{cases}
T_w T_{w'} = T_{ww'} & \text{if $\ell(ww')=\ell(w)+\ell(w')$,}\\
T_s^2 = 1 + (v^{\ph(s)}-v^{-\ph(s)}) T_s & \text{if $s \in S$.}
\end{cases}\]
The Bruhat--Chevalley order on $W$ will be denoted by $\bruhatordre$. 

\begin{rema} \label{eq:multiplication-g}
If $\PC$ is an assertion, then we define $\d_\PC$ by $\d_\PC=1$ (resp. $0$) 
if $\PC$ is true (resp. false). For instance, $\d_{i\egal j}$ replaces the 
usual Kronecker symbol $\d_{i,j}$. With this notation, we have 
\[\renewcommand{\arraystretch}{1.4} \begin{array}{l} 
T_sT_w=T_{sw} + \d_{sw \bruhatordrestrict w} (v^{\ph(s)}-v^{-\ph(s)}) T_w,\\
T_wT_s=T_{ws} + \d_{ws \bruhatordrestrict w} (v^{\ph(s)}-v^{-\ph(s)}) T_w
\end{array}\]
for all $s \in S$ and $w \in W$. 
\end{rema}

\medskip

\begin{lem}\label{lem:invo-central}
Let $\CC$ be a union of conjugacy classes of involutions in $W$. Then 
\[ \Tgras{\CC}:=\sum_{w \in \CC} T_w \quad \mbox{is central in $\HC$}.\]
\end{lem}

\begin{proof}
Since $(T_s)_{s \in S}$ generates the $A$-algebra $\HC$, it is sufficient 
to show that $T_s \Tgras{\CC} = \Tgras{\CC} T_s$ for all $s \in S$. But, 
by Remark~\ref{eq:multiplication-g}, we have 
\begin{align*}
T_s \Tgras{\CC} &= \sum_{w \in \CC} T_{sw} + 
\sum_{w \in \CC} \d_{sw \bruhatordrestrict w} (v^{\ph(s)}-v^{-\ph(s)}) T_w,\\
\Tgras{\CC} T_s &= \sum_{w \in \CC} T_{ws} + \sum_{w \in \CC} 
\d_{ws \bruhatordrestrict w} (v^{\ph(s)}-v^{-\ph(s)}) T_w.
\end{align*}
Now, as $\CC$ is a union of conjugacy classes, we have $s\CC=\CC s$. Moreover, 
as elements of $\CC$ are involutions, we have $sw \bruhatordrestrict w$ if 
and only if $ws \bruhatordrestrict w$ (for any $w \in \CC$). The result 
follows.
\end{proof}

If $\CC$ is a conjugacy class of reflections, the above result is stated
in \cite[Exp.~3.3.8]{geja}; in this case, it is due to L. Iancu 
(unpublished).

\medskip

\begin{rema} \label{main1a} Let $f \colon W \rightarrow \RM$ be any 
class function on $W$. Let $\CC$ be a union of conjugacy classes of
involutions in $W$. Then we also have that 
\[\Tgras{\CC}^f:=\sum_{w \in \CC} f(w)T_w\quad \mbox{is central in $\HC$}.\]
(Indeed, it is sufficient to prove this in the case where $\CC$ is a 
single conjugacy class in which case we have $\Tgras{\CC}^f=f(t)\Tgras{\CC}$
where $t \in \CC$ is fixed.) In particular, applying this to the sign 
character $\varepsilon$ of $W$, we obtain 
\[\Tgras{\CC}^\varepsilon=\sum_{w \in \CC} (-1)^{\ell(w)}T_w
\quad \mbox{is central in $\HC$}.\]
\end{rema}

\medskip

For any $a=\sum_{\g \in \RM} a_\g v^\g$, we set $\aov=\sum_{\g \in \G} a_\g 
v^{-\g}$. This can be extended to an antilinear automorphism $\HC 
\rightarrow \HC$, $h \mapsto \bar{h}$, by the formula 
\[\overline{\sum_{w \in W} a_w T_w}=\sum_{w \in W} \aov_w T_{w^{-1}}^{-1}.\]
We set $A_{<0}=\oplus_{\g < 0} ~A v^\g$ and $\HC_{<0}=\oplus_{w \in W} ~A_{<0} 
T_w$. By~\cite[Theorem~5.2(a)]{lusztig}, there exists a unique $A$-basis
$(c_w)_{w \in W}$ of $\HC$, called the {\it Kazhdan--Lusztig basis}, such 
that 
\[ \begin{cases} \cov_w=c_w,\\ c_w \equiv T_w \mod \HC_{< 0}.\end{cases}\]
We now define $\prel$ (resp $\prer$, resp.  $\prelr$) as the coarsest 
preorder such that, for all $w \in W$, $\oplus_{y \prel w} ~A c_y$ 
(resp.  $\oplus_{y \prer w} ~A c_y$, resp.  $\oplus_{y \prelr w}~ A c_y$) 
is a left (resp. right, resp. two-sided) ideal of $\HC$. We define $\siml$ 
(resp.  $\simr$, resp. $\simlr$) as the equivalence relation associated 
with $\prel$ (resp. $\prer$, resp.  $\prelr$): its equivalence classes are 
called the {\it left} (resp.  {\it right}, resp. {\it two-sided}) 
{\it cells}. We denote by $\Cell_L(W)$ (resp. $\Cell_R(W)$, resp. 
$\Cell_{LR}(W)$) the set of left (resp. right, resp. two-sided) cells of $W$. 

In order to define the corresponding {\it cell modules} it will be
convenient, as in the later chapters of \cite{lusztig}, to work with a 
slightly modified version of the basis $(c_w)_{w\in W}$. Let $h \mapsto 
h^\dagger$ denote the unique $A$-algebra automorphism of $\HC$ such that 
\[ T_s^\dagger=-T_s^{-1} \qquad \mbox{for all $s \in W$}.\]
(See \cite[3.5]{lusztig}). Then, clearly, $(c_w^\dagger)_{w \in W}$ also 
is an $A$-basis of $\HC$. 

\medskip
\begin{rema} \label{eq:cwdagger-tw} By \cite[Theorem~5.2(b)]{lusztig}, we
have 
\[ c_w \equiv T_w \mod\Bigl(\mathop{\oplus}_{y \bruhatordrestrict w} A_{<0} 
T_y\Bigr) \qquad \mbox{for all $w\in W$}.\]
Since $c_w=\cov_w$, we also have 
\[ c_w^\dagger \equiv (-1)^{\ell(w)}T_w \mod\Bigl(\mathop{\oplus}_{y 
\bruhatordrestrict w} A_{>0} T_y\Bigr) \qquad \mbox{for all $w \in W$}.\]
\end{rema}

\medskip
Now, for every left cell $C$, we can construct a left $\HC$-module $V_C$,
called a {\it left cell module}, as follows. For $x,y \in W$, let us write
\[ c_xc_y=\sum_{z \in W} h_{x,y,z} c_z \qquad \mbox{where} \qquad
h_{x,y,z} \in A.\]
Then, as an $A$-module, $V_C$ is free with a basis $\{e_x \mid x \in 
C\}$. The action of $\HC$ on $V_C$ is given by the formula 
(see \cite[21.1]{lusztig}):
\[ c_x^\dagger.e_y=\sum_{z \in C} h_{x,y,z} e_z\qquad \mbox{where $x \in W$
and $y \in C$}.\]
We can perform similar constructions for right and two-sided
ideals, giving rise to right $\HC$-modules and $(\HC,\HC)$-bimodules,
respectively.

Now, let $K$ denote the fraction field of $A$ and, if $M$ is an $A$-module, 
let $KM=K\otimes_A M$. Then it is well-known (see, for example, 
\cite[9.3.5]{gepf}) that the $K$-algebra $K\HC$ is split and semisimple 
so, by Tits' deformation Theorem, there is a bijection 
\[\begin{array}{ccc}
\Irr(W) & \longiso & \Irr(K\HC) \\
\chi & \longmapsto & \chi_\ph.
\end{array}\]
Here, $\chi$ can be retrieved from $\chi_\ph$ through the specialization 
$v^\g \mapsto 1$. 

\medskip
\begin{defi}[\protect{\cite{KL}}, \protect{\cite{l}}] \label{def:fam} 
We define a partition of $\Irr(W)$, depending on $\ph$, as follows. For 
a two-sided cell $\CG$, we denote by $\Irr_\CG(W)$ the set of 
irreducible characters $\chi$ of $W$ such that $\chi_\ph$ is an irreducible 
constituent of $KV_C$, where $C$ is a left cell contained in $\CG$.
Then:
\[\Irr(W)=\coprod_{\CG \in \Cell_{LR}(W)} \Irr_\CG(W).\]
Note that, for each two-sided cell $\CG$, we have
\[ |\CG|=\sum_{\chi \in \Irr_\CG(W)} \chi(1)^2.\]
\end{defi}

If $C$ is a left cell, we denote by $\isomorphisme{C}$ the character 
of $W$ obtained by specialization through $v^\g \mapsto 1$ from the 
character of $K\HC$ afforded by $V_C$. An indication of the connection
between left cells and involutions is given by the following result.

\medskip
\begin{prop}[\protect{\cite{geck invo}}] \label{myinv} Let $C$ be a left 
cell in $W$. Then the number of involutions in $C$ is equal to the number
of irreducible constituents of $[C]$ (counting multiplicities).
\end{prop}

We denote by $\GC_\ph(W)$ the 
following graph: its vertices are the irreducible characters of $W$ and 
two irreducible characters $\chi$ and $\chi'$ are joined by an edge if 
there exists a left cell $C$ such that $\chi$ and $\chi'$ are irreducible 
components of $\isomorphisme{C}$. In order to relate the graph
$\GC_\ph(W)$ to the partition of $\Irr(W)$ in Definition~\ref{def:fam}
we need the following result.

\medskip

\begin{prop}[\protect{\cite[Theorem~12.15]{LuB}} and 
\protect{\cite[Corollary~3.9]{geck invo}}] 
\label{lem:gamma-smooth} Let $C$ and $C'$ be two left cells. Then:
$\langle \isomorphisme{C},\isomorphisme{C'} \rangle_W = 
|C' \cap C^{-1}|$. 
\end{prop}

\medskip

As two-sided cells are unions of left cells, the sets $\Irr_{\CG}(W)$
are unions of connected components of the graph $\GC_\ph(W)$. It is 
conjectured that the converse holds:

\medskip

\begin{coro}\label{coro:graphe}
{\bfit Assume that Lusztig's Conjectures P's for $(W,S,\ph)$ hold}. Then the 
sets $\Irr_{\CG}(W)$ are the connected components of the graph $\GC_\ph(W)$. 
\end{coro}

\begin{proof}
Indeed, if Lusztig's Conjectures P's for $(W,S,\ph)$ hold, then $\simlr$ 
is the equivalence relation generated by $\siml$ and $\simr$; 
see~\cite[\S{14.2},~Conjecture~P9]{lusztig}. So the result follows from 
Proposition~\ref{lem:gamma-smooth}. 
\end{proof}

\medskip

We shall also need the following result whose proof relies on some
case--by--case arguments and explicit computations.

\medskip

\begin{prop}[\protect{\cite[Chap.~22]{lusztig}}] \label{mult1} 
{\bfit Assume that Lusztig's Conjectures P's for $(W,S,\ph)$ hold}. 
Let $\chi \in \Irr(W)$. Then there exists a left cell $C$ of $W$
such that $\langle [C],\chi\rangle_W=1$.
\end{prop}

\begin{proof} By the explicit results in \cite[\S 22]{lusztig} (see also
\cite[\S 7]{geck plus} and the references there for the non-crystallographic
types), every $\chi \in \Irr(W)$ appears with multiplicity $1$ in some 
``contructible'' character, as defined in \cite[22.1]{lusztig}. (For Weyl
groups and the equal parameter case, this statement already appeared in 
\cite[5.30]{LuB}.) On the other hand, since Lusztig's Conjectures P's 
for $(W,S,\ph)$ are assumed to hold, we can apply \cite[Lemma~22.2]{lusztig}
which shows that every constructible character is of the form $[C]$ for
some left cell $C$. 
\end{proof}

\section{Leading coefficients} \label{sec:lc}

Lusztig has associated with any $\chi \in \Irr(W)$ two invariants 
$\ab_\chi\in \Gamma_{\geq 0}$ and $f_\chi\in \RM_{>0}$; see 
\cite[Chap.~4]{LuB}, \cite{lulc}, \cite[\S{20}]{lusztig}. Let us briefly
recall how this is done. It is known that $\chi_\varphi(T_w) \in A$ for 
all $w \in W$; see \cite[9.3.5]{gepf}. Thus, we can define
\[ \ab_\chi:=\min\{\gamma\in \G_{\geq 0} \mid v^\gamma\, \chi_\varphi(T_w)
\in A_{\geq 0} \mbox{ for all $w \in W$}\}.\]
Consequently, there are unique numbers $c_{w,\chi} \in \RM$
($w \in W$) such that
\[v^{\ab_\chi}\,\chi_\varphi(T_w) \equiv (-1)^{\ell(w)}\,c_{w,\chi} \bmod
A_{>0}.\]
These numbers are Lusztig's ``leading coefficients of character values'';
see \cite{LuB}, \cite{lulc}. Since $\chi_\varphi(T_w)= 
\chi_\varphi(T_{w^{-1}})$ for all $w \in W$ (see 
\cite[8.2.6]{gepf}), we certainly have
\[ c_{w,\chi}=c_{w^{-1},\chi} \qquad \mbox{for all $w \in W$}.\]
Given $\chi$, there is at least one $w \in W$ such that $c_{w,\chi}\neq 0$
(by the definition of $\ab_\chi$). Hence, the sum of all $c_{w,\chi}^2$
($w\in W$) will be strictly positive and so we can write that sum as
$f_\chi\,\chi(1)$ where $f_\chi\in \RM$ is strictly positive. We have 
the following orthogonality relations (see \cite[Exc.~9.8]{gepf}):
\[ \sum_{w \in W} c_{w,\chi}\, c_{w,\chi'}= \left\{\begin{array}{cl}
f_\chi \chi(1) & \quad \mbox{if $\chi=\chi'$},\\ 0 & \quad \mbox{otherwise}.
\end{array}\right.\]
The coefficients $c_{w,\chi}$ and the numbers $f_\chi$ are
related to the left and two-sided cells of $W$. We shall now state a
few results which make this relation more precise.

\medskip

\begin{prop}[\protect{\cite[5.8]{LuB}} and 
\protect{\cite[3.8]{geck invo}}] \label{refine1} 
Let $C$ be a left cell and $\chi,\chi' \in \Irr(W)$. Then
\[ \sum_{w \in C} c_{w,\chi}\,c_{w,\chi'}=\left\{\begin{array}{cl} f_\chi\,
\langle [C],\chi\rangle_W & \quad \mbox{if $\chi=\chi'$},\\ 0 & \quad
\mbox{otherwise}.\end{array} \right.\]
\end{prop}

\medskip

%\begin{prop}[\protect{\cite[Theorem~1.1]{geck invo}}] \label{refine3}
%Let $C$ be a left cell. Then 
%\[ \sum_{\chi \in \Irr(W)} \langle [C],\chi\rangle_W=
%|\{ w\in C \mid w^2=1\}|.\]
%\end{prop}
%
%\medskip

\begin{coro} \label{refine0} Let $\chi \in \Irr(W)$ and $w \in W$. 
If $c_{w,\chi}\neq 0$ then $\langle [C],\chi\rangle_W \neq 0$ where 
$C$ is the left cell containing~$w$. In particular, $\chi \in \Irr_{\CG}(W)$
where $\CG$ is the two-sided cell such that $w \in \CG$.
\end{coro}

\begin{proof} If $c_{w,\chi}\neq 0$ and $w \in C$, then the left hand
side of the formula in Proposition~\ref{refine1} (where $\chi'=\chi$) 
is non-zero. Hence, so is the right hand side, that is, $\langle [C],
\chi\rangle_W\neq 0$. 
\end{proof}

\medskip

\begin{exemple} \label{lem:useful} Let $W'\subseteq W$ be a standard
parabolic subgroup, $\varepsilon'$ the sign character of $W'$ and 
$w_0'\in W'$ the longest element in $W'$. Let $\chi \in \Irr(W)$ 
be such that 
\[ \ab_\chi=\ph(w_0') \qquad \mbox{and} \qquad \big\langle
\Ind_{W'}^W(\varepsilon'),\chi\big\rangle_W\neq 0.\]
Then $\chi \in \Irr_{\CG}(W)$ where $\CG$ is the two-sided cell
which contains $w_0'$. (Indeed, by \cite[Cor.~2.8.6]{geja}, we have 
\[ c_{w_0',\chi}=\pm \big\langle \Ind_{W'}^W(\varepsilon'),\chi
\big\rangle_W\neq 0.\]
and it remains to use Corollary~\ref{refine0}.)
\end{exemple}

\medskip

\begin{defi} \label{distinv}
We define the set of ``{\it distinguished elements}'' in $W$ by 
\[ \DC:=\{ w\in W \mid n_w \neq 0\} \qquad \mbox{where} \qquad 
n_w:=\sum_{\chi \in \Irr(W)} f_\chi^{-1}\, c_{w,\chi}.\]
(Note that $\DC$ depends on $\ph$.) If Lusztig's Conjectures P's for
$(W,S,\ph)$ hold, then \cite[Lemma~3.7]{myp115} shows that this definition 
coincides with that in \cite[14.1]{lusztig}. In particular, by Conjectures
P5 and P6, we have $n_d=\pm 1$ and $d^2=1$ for all $d \in \DC$; 
furthermore, by P13, every left cell contains a unique element of $\DC$. 
\end{defi}

\medskip

\begin{prop}\label{refine2} 
{\bfit Assume that Lusztig's Conjectures P's for $(W,S,\ph)$ hold}. 
Let $C$ be a left cell and $\DC \cap C=\{d\}$. Then 
\[ c_{d,\chi}= n_d\langle [C],\chi\rangle_W \qquad \mbox{and}\qquad  
\sum_{\chi \in \Irr(W)} f_\chi^{-1}\langle [C],\chi\rangle_W=1.\]
\end{prop}

\begin{proof} The first identity is contained in \cite[20.6, 
21.4]{lusztig}. Then the second identity immediately follows 
from the above formula for $n_d$.
\end{proof}

\medskip

\begin{defi}\label{defi:smooth}
A two-sided cell $\CG$ is said to be {\bfit smooth} if $|\Irr_\CG(W)|=1$. The set 
of smooth two-sided cells will be denoted by $\Cell_{LR}^{\mathrm{smooth}}(W)$. 
\end{defi}

The next result gives a characterization of smooth two-sided cells whenever 
Lusztig's Conjectures P's hold:

\medskip

\begin{lem}\label{lem:smooth}
{\bfit Assume that Lusztig's Conjectures P's for $(W,S,\ph)$ hold}.
Let $\CG$ be a two-sided cell. We denote $\CG_{(2)}=\{w \in \CG 
\mid w^2=1\}$. Then the following are equivalent:
\begin{itemize}
\itemth{1} $\CG$ is ``smooth'', that is, $|\Irr_\CG(W)|=1$.

\itemth{2} 
There exists a left cell $C \subseteq \CG$ such that $\isomorphisme{C} \in
\Irr(W)$.

\itemth{3} $f_\chi = 1$ for some $\chi \in \Irr_\CG(W)$.

\itemth{4} For any left cell $C \subseteq \CG$, we have $\isomorphisme{C} 
\in \Irr(W)$.

\itemth{5} $|\CG|=|\CG_{(2)}|^2$.

\itemth{6} $\CG_{(2)}\subseteq \DC$, that is, all involutions in $\CG$
are ``distinguished''.
\end{itemize}
\end{lem}

Note also that the condition ``$[C] \in \Irr(W)$'' can be replaced by
``$|C \cap C^{-1}|=1$''; see Proposition~\ref{lem:gamma-smooth}. 

\begin{proof} First we show the equivalence of (1), (2), (3), (4).

``$(1) \Rightarrow (2)$'' Let $\Irr_{\CG}(W)=\{\chi\}$. Let $C$ be a left 
cell as in Proposition~\ref{mult1}. Since $\langle [C],\chi\rangle_W\neq 0$,
we have $C \subseteq \CG$; see Definition~\ref{def:fam}. Furthermore, by 
Corollary~\ref{coro:graphe}, {\it every} irreducible constituent of $[C]$ 
belongs to $\CG$. Hence, $\chi$ is the only constituent of $[C]$. Since it 
occurs with multiplicity~$1$, we have $[C]=\chi \in \Irr(W)$.

``$(2) \Rightarrow (3)$'' If $\chi:=[C] \in \Irr(W)$, then the identity
in Proposition~\ref{refine2} reduces to $1=f_{\chi}^{-1}$ and so $f_\chi=1$.

``$(3) \Rightarrow (4)$'' Let $C$ be a left cell as in 
Proposition~\ref{mult1}. Then, as above, we have $C \subseteq \CG$.
The identity in Proposition~\ref{refine2} now shows that
\[ 1=1+\sum_{\chi \neq \psi \in \Irr(W)} f_\psi^{-1}\langle [C],
\psi\rangle_W.\]
Hence, we have $\langle [C],\psi\rangle_W=0$ for all $\psi \neq \chi$
and so $[C]=\chi\in \Irr(W)$. Now let $C'$ be another left cell contained 
in $\CG$. By Corollary~\ref{coro:graphe}, there exists a sequence $C=C_0, 
C_1,\ldots, C_n=C'$ of left cells contained in $\CG$ such that $\langle [C_i],
[C_{i+1}]\rangle_W\neq 0$ for all $i$. We shall prove by induction on $i$
that $\isomorphisme{C_i} =\isomorphisme{C}$. This is clear if $i=0$, so 
assume that $\isomorphisme{C_i} =\isomorphisme{C}$ and let us show that 
$\isomorphisme{C_{i+1}} =\isomorphisme{C}$. By assumption, we have
$\langle \isomorphisme{C_i},\isomorphisme{C_{i+1}} \rangle_W \neq 0$,
which means that $\langle \isomorphisme{C_i},\psi\rangle_W \leq 
\langle \isomorphisme{C_{i+1}},\psi\rangle_W$ for all $\psi \in \Irr(W)$. 
Applying the identity in Proposition~\ref{refine2} to both $C_i$ and 
$C_{i+1}$, we obtain
\[ 1=\sum_{\psi\in \Irr(W)}f_\psi^{-1}\langle [C_i],\psi\rangle_W
\leq \sum_{\psi\in \Irr(W)}f_\psi^{-1}\langle [C_{i+1}],\psi\rangle_W=1.\]
Hence, we must have $\langle \isomorphisme{C_i},\psi\rangle_W= 
\langle \isomorphisme{C_{i+1}}\rangle_W$ for all $\psi \in \Irr(W)$
and so $\isomorphisme{C_{i+1}}=\isomorphisme{C_i}\in \Irr(W)$, as required. 
Thus, (4) holds.

``$(4) \Rightarrow (1)$'' By Corollary~\ref{coro:graphe}, we necessarily
have $\chi:=[C]=[C']$ for all left cells $C,C' \subseteq \CG$ and then
$\Irr_{\CG}(W)=\{\chi\}$. 

Now we show the remaining equivalences.

``$(1) \Leftrightarrow (5)$'' Let $|\Irr_{\CG}(W)|=n\geq 1$ and write 
$\Irr_{\CG}(W)=\{\chi_1,\ldots, \chi_n\}$. Then, as noted in 
Definition~\ref{def:fam}, we have 
\[ |\CG|=\chi_1(1)^2+\cdots + \chi_n(1)^2.\]
On the other hand, it easily follows from Proposition~\ref{myinv} that
 $|\CG_{(2)}|=\chi_1(1)+\cdots +  \chi_n(1)$; see \cite[Cor.~3.12]{geck 
invo}, Hence, we have
\[ |\CG_{(2)}|^2=\bigl(\chi_1(1)+\cdots +  \chi_n(1)\bigr)^2,\]
which implies that $|\CG|=|\CG_{(2)}|^2$ if and only if $n=1$.

``$(4) \Leftrightarrow (6)$'' Recall that, by Lusztig's Conjecture P13, 
every left cell contains a unique element of $\DC$;  furthermore, by P6,
we have $d^2=1$ for all $d \in \DC$. So the equivalence immediately
follows from Proposition~\ref{myinv}.
\end{proof}

\begin{table}[htbp]
\centerline{\small \begin{tabular}{@{{{\vrule width 1.5pt}}}c@{{{\vrule 
width 1.5pt}}}c|c@{{{\vrule width 1.5pt}}}} \hlinewd{1.5pt}
\petitespace \hskip2mm Type of $W$ \hskip2mm & 
\hskip2mm $|\Cell_{LR}(W)|$ \hskip2mm & 
\hskip2mm $|\Cell_{LR}^{\text{smooth}}(W)|$ \hskip2mm~ \\
\hlinewd{1.5pt}
%\petitespace
%\vphantom{${\DS{\frac{\DS{A^A}}{}}}$}
$I_2(m)$ & 3 & 2 \\ \hline 
$B_3$ & 6 & 4 \\
$B_4$ & 10 & 5 \\
$B_5$ & 16 & 6 \\ 
$B_6$ & 26 & 10 \\ 
$B_7$ & 40 & 12 \\ 
$B_8$ & 60 & 15 \\ \hline
$D_4$ & 11 & 10 \\
$D_5$ & 14 & 12 \\
$D_6$ & 27 & 22 \\ 
$D_7$ & 35 & 25 \\ 
$D_8$ & 60 & 40 \\  \hline
$E_6$ & 17 & 14 \\
$E_7$ & 35 & 24 \\
$E_8$ & 46 & 23 \\ \hline
$F_4$ & 11 & 8 \\ \hline
$H_3$ & 7 & 4 \\
$H_4$ & 13 & 6 \\
\hlinewd{1.5pt}
\end{tabular}}
\bigskip
\caption{Number of smooth cells (equal parameters)}\label{table:eq}
\end{table}

\medskip

\begin{exemple} \label{p115eq} Assume that we are in the equal
parameter case where $\ph$ is constant. In this case, it is known
that Lusztig's Conjectures P's for $(W,S,\ph)$ hold; see \cite{Lu2}, 
\cite[Chap.~15]{lusztig} (for Weyl groups) and \cite{fokko} (for the 
remaining types). 

Note that ``smooth'' two-sided cells actually occur quite often in this case.
For example, assume that $(W,S)$ is of type $A_{n-1}$ where $W=\SG_n$
is the symmetric group. Then we are automatically in the equal parameter
case and we have $f_\chi=1$ for all $\chi \in \Irr(W)$; see, for example,
\cite[5.16]{LuB} and \cite[9.4.5]{gepf}. Hence, every two-sided cell in
$W$ is smooth in this case.

For further information, we give in Table~\ref{table:eq} the number of 
smooth two-sided cells (equal parameter case) whenever 
$|S| \le 8$ and $(W,S)$ is not of type $A$. To compute this table
it suffices, by Lemma~\ref{lem:smooth}, to find all $\chi \in \Irr(W)$ 
such that $f_\chi=1$, and this information is easily available from
the tables in \cite{LuB}, \cite[Appendix]{gepf}.
\end{exemple}

\medskip

\begin{exemple} \label{smoothb} Let $(W,S)$ be of type $B_n$ and write
$S=\{t,s_1,s_2,\dots,s_{n-1}\}$ in such a way that the Dynkin diagram
of $(W,S)$ is given as follows.
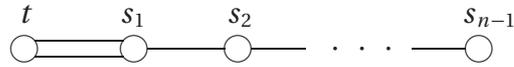
\begin{center}
\begin{picture}(220,30)
\put( 40, 10){\circle{10}}
\put( 44,  7){\line(1,0){33}}
\put( 44, 13){\line(1,0){33}}
\put( 81, 10){\circle{10}}
\put( 86, 10){\line(1,0){29}}
\put(120, 10){\circle{10}}
\put(125, 10){\line(1,0){20}}
\put(155,  7){$\cdot$}
\put(165,  7){$\cdot$}
\put(175,  7){$\cdot$}
\put(185, 10){\line(1,0){20}}
\put(210, 10){\circle{10}}
\put( 38, 20){$t$}
\put( 76, 20){$s_1$}
\put(116, 20){$s_2$}
\put(204, 20){$s_{n{-}1}$}
\end{picture}
\end{center}
We set $\ph(t)=b$ and $\ph(s_1)=\cdots=\ph(s_{n-1})=a$. Then it follows 
from the computation of constructible characters 
in~\cite[Proposition~22.25]{lusztig} that:
\begin{equation*}
\mbox{\it $f_\chi=1$ for all $\chi \in \Irr(W)$} \qquad
\Longleftrightarrow \qquad b \not\in \{a,2a,\dots,(n-1)a\}.\tag{a}
\end{equation*}
Hence, if Lusztig's Conjectures P's for $(W,S,\ph)$ hold, then 
Lemma~\ref{lem:smooth} shows that all two-sided cells of $W$ are smooth if 
and only if $b \not\in \{a,2a,\dots,(n-1)a\}$.
Without assuming that Lusztig's Conjectures P's for $(W,S,\ph)$ hold,
the only known results are the following:
\begin{equation*}
\mbox{\it All the two-sided cells in $W$ are smooth if $a=2b$ or 
$3a=2b$ or $b > (n-1) a$}. \tag{b}
\end{equation*}
If $a=2b$ or $3a=2b$, then (b) follows essentially from 
\cite[\S{16}]{lusztig} (see~\cite[Theorem~3.14]{BGIL} for some 
explanation). If $b > (n-1) a$, then (b) follows 
from~\cite[Theorem~7.7]{bonnafe iancu} and 
\cite[Theorem~3.5~and~Corollary~5.2]{bonnafe two}. 
\end{exemple}

\section{A basic identity} \label{sec:invspec}

\springer{{\bf Hypothesis.} {\it Throughout this section we assume 
that Lusztig's Conjectures P's hold for $(W,S,\ph)$.}}

\medskip

The main result of this section is the following basic identity, which links
cells and involutions through the leading coefficients of character values. 

\begin{lem}[The $(\CG,C,\CC)$-identity] \label{main1} Let $\CG$ be a 
two-sided cell and $C$ a left cell contained in $\CG$. Let $\CC$ be a 
union of conjugacy classes of involutions in $W$. Then
\[\langle [C],\chi\rangle_W\sum_{w \in \CC \cap \CG} c_{w,\chi}=\chi(1) 
\sum_{w \in \CC \cap C} c_{w,\chi} \qquad \mbox{for all $\chi \in 
\Irr(W)$}.\]
\end{lem}

\begin{proof} Let $\Zrm(\HC)$ be the centre of $\HC$. We denote by 
$\o_\chi : \Zrm(\HC) \to A$ the {\it central character} associated with 
$\chi_\ph$: if $z \in \Zrm(\HC)$, then $\o_\chi(z)=\chi_\ph(z)/\chi(1)$. 
Now consider the central element
\[\Tgras{\CC}^\varepsilon=\sum_{w \in \CC} (-1)^{\ell(w)}T_w \qquad
\mbox{(see Remark~\ref{main1a})}.\]
The desired identity will be obtained by evaluating
$\chi_\ph$ on $\Tgras{\CC}^\varepsilon T_d$, where $d$ is the unique element 
of $\DC$ contained in $C$ (see Lusztig's Conjecture P13). First note that,
if $\chi\not\in \Irr_{\CG}(W)$, then both sides of the identity are zero;
see Corollary~\ref{refine0}. 

We can now assume that $\chi\in \Irr_{\CG}(W)$. Since $\Tgras{\CC}^\varepsilon
\in \Zrm(\HC)$, we have $\chi(\Tgras{\CC}^\varepsilon)=\chi(1) 
\omega_\chi(\Tgras{\CC}^\varepsilon)$ and $\chi(\Tgras{\CC}^\varepsilon T_d)
=\omega_\chi(\Tgras{\CC}^\varepsilon)\chi(T_d)$. Furthermore, 
\[v^{\ab_\chi}\chi(\Tgras{\CC}^\varepsilon)=\sum_{w \in \CC} 
v^{\ab_\chi}(-1)^{\ell(w)} \chi(T_w) \equiv \Bigl(\sum_{w \in \CC} 
c_{w,\chi}\Bigr)\bmod A_{>0}.\]
It follows that 
\[v^{2\ab_\chi} \chi(1)\chi(\Tgras{\CC}^\varepsilon T_d)\equiv 
\bigl(v^{\ab_\chi} \chi(\Tgras{\CC}^\varepsilon) \bigr)\bigl(v^{\ab_\chi}
\chi(T_d) \bigr)\equiv (-1)^{\ell(d)}\Bigl(\sum_{w \in \CC} c_{w,\chi}
c_{d,\chi}\Bigr)\bmod A_{>0}.\]
Now, by Proposition~\ref{refine2}, we have $c_{d,\chi}=n_d \langle [C],
\chi\rangle_W$. Thus, we obtain
\[v^{2\ab_\chi} \chi(1)\chi(\Tgras{\CC}^\varepsilon T_d)\equiv (-1)^{\ell(d)}n_d
\langle [C],\chi\rangle_W \Bigl(\sum_{w \in \CC} c_{w,\chi}\Bigr) 
\bmod A_{>0}.\]
The summation on the right hand side can be taken over all $w \in 
\CC \cap \CG$ (instead of $w\in \CC$) since $c_{w,\chi}=0$ unless 
$w \in \CG$; see Corollary~\ref{refine0}. Next we re-write 
$\Tgras{\CC}^\varepsilon T_d$ using the Kazhdan--Lusztig basis. For any 
$w \in W$, we have $T_w \equiv (-1)^{\ell(w)}c_w^\dagger \bmod \HC_{>0}$; see
Remark~\ref{eq:cwdagger-tw}.  This yields 
\begin{align*}
\Tgras{\CC}^\varepsilon T_d&=\sum_{w \in \CC} (-1)^{\ell(w)}T_wT_d \in 
\sum_{w \in W} (c_w^\dagger+\HC_{>0})((-1)^{\ell(d)}c_d^\dagger+\HC_{>0})\\&
\subseteq \Bigl(\sum_{w \in \CC} (-1)^{\ell(d)}c_w^\dagger 
c_d^\dagger\Bigr)+ \HC_{>0}\HC_{\geq 0}+\HC_{\geq 0}\HC_{>0}.
\end{align*}
We certainly have $v^{\ab_\chi}\chi(h) \in A_{\geq 0}$ for any
$h \in \HC_{\geq 0}$ and $v^{\ab_\chi}\chi(h) \in A_{>0}$ for any
$h \in \HC_{>0}$. Hence, we obtain
\[ v^{2\ab_\chi}\chi(\Tgras{\CC}^\varepsilon T_d) \equiv (-1)^{\ell(d)}
\Bigl(\sum_{w \in \CC} v^{2\ab_\chi}\chi\bigl(c_w^\dagger c_d^\dagger 
\bigr)\Bigr) \bmod A_{>0}.\]
We now look at the term $\chi(c_w^\dagger c_d^\dagger)$ (for $w\in
\CC$) in more detail. Following Lusztig's notation 
in~\cite[\S{13}]{lusztig}, we set for $x$, $y \in W$, 
\[c_xc_y = \sum_{z \in W} h_{x,y,z} c_z \qquad \mbox{where $h_{x,y,z} 
\in A$}.\]
Furthermore, if $z \in W$, we define $\ab(z)=\max\{\deg(h_{x,y,z})
\mid x,y \in W\}$. Let us now consider 
\[ v^{2\ab_\chi}\chi(c_w^\dagger c_d^\dagger)= \sum_{x \in W} 
\bigl(v^{\ab_\chi} h_{w,d,x}\bigr) \bigl(v^{\ab_\chi} \chi(c_x^\dagger) 
\bigr).\]
Let $x \in W$ be such that $h_{w,d,x}\neq 0$ and $\chi(c_x^\dagger)\neq 0$.
Since Lusztig's Conjecture P4 holds, the first condition implies that 
$\ab(d)\leq \ab(x)$. By \cite[Lemma~3.5]{myp115}, the second condition 
implies that $\ab(x)\leq \ab_\chi$. (Note that the function $\tilde{a}(w)$ 
in \cite[3.5]{myp115} agrees with $\ab(w)$ by \cite[Prop.~3.6 and 
Rem.~4.2]{myp115}.) On the other hand, since $\chi\in \Irr_{\CG}(W)$, we 
have $\ab_\chi=\ab(d)$; see \cite[Proposition~20.6]{lusztig}. Thus, we 
must have $\ab(d)=\ab(x)=\ab_\chi$. Furthermore, since $h_{w,d,x}\neq 0$
and  $\ab(d)= \ab(x)$, we can now even conclude that $x \in C$, by 
Lusztig's Conjecture P9. Thus, we obtain
\[v^{2\ab_\chi}\chi(c_w^\dagger c_d^\dagger)=\sum_{x \in C} \bigl(v^{\ab(x)}
h_{w,d,x}\bigr) \bigl(v^{\ab_\chi} \chi(c_x^\dagger) \bigr).\]
Now, by Remark~\ref{eq:cwdagger-tw}, we have $v^{\ab_\chi}\,
\chi_\varphi(c_w^\dagger) \equiv c_{w,\chi} \bmod A_{>0}$.
Hence, taking constant terms in the above identity, we obtain
\[v^{2\ab_\chi}\chi(c_w^\dagger c_d^\dagger) \equiv 
\Bigl(\sum_{x \in C} \gamma_{w,d,x^{-1}}\, c_{x,\chi}\Bigr) \bmod A_{>0};\]
here, we denote by $\gamma_{w,d,x^{-1}}$ the constant term of 
$v^{\ab(x)}h_{w,d,x}$, as in \cite[13.6]{lusztig}. By Lusztig's
Conjectures P2, P5 and P7, we have
\[ \gamma_{w,d,x^{-1}}=\left\{\begin{array}{cl} n_d & \quad \mbox{if
$x=w$},\\ 0 & \quad \mbox{otherwise}.\end{array}\right.\]
We conclude that $v^{2\ab_\chi}\chi(c_w^\dagger c_d^\dagger) \equiv 
\delta_{w \in C} n_d\, c_{w,\chi} \bmod A_{>0}$ and so
\begin{align*} 
v^{2\ab_\chi}\chi(\Tgras{\CC}^\varepsilon T_d) &\equiv (-1)^{\ell(d)}
\sum_{w \in \CC} v^{2\ab_\chi}\chi\bigl(c_w^\dagger c_d^\dagger \bigr) \\
&\equiv (-1)^{\ell(d)}n_d\sum_{w \in \CC} \delta_{w \in C} n_d\, c_{w,\chi} 
\\ & \equiv (-1)^{\ell(d)}n_d\Bigl(\sum_{w \in \CC \cap C} c_{w,\chi}\Bigr)
\bmod A_{>0}.
\end{align*}
Comparing with our earlier expression for $v^{2\ab_\chi} \chi(1)
\chi(\Tgras{\CC}^\varepsilon T_d) \bmod A_{>0}$ yields the desired identity.
\end{proof}

\medskip

\begin{exemple} \label{smooth1} Let $\CG$ be a two-sided cell which is 
``smooth'', that is, we have  
\[ \Irr_{\CG}(W)=\{\chi\} \qquad \mbox{for some $\chi \in \Irr(W)$}.\]
Let $d \in \CG \cap \DC$ and $\CC$ a union of conjugacy classes of
involutions in $W$. Let $C$ be the left cell containing $d$. Then we claim 
that the $(\CG,C,\CC)$-identity in Lemma~\ref{main1} reduces to:
\[ \sum_{w \in \CC \cap \CG} c_{w,\chi}= \left\{\begin{array}{cl} 
\chi(1)n_d & \quad \mbox{if $d \in \CC$},\\
0 & \quad \mbox{otherwise}.\end{array}\right.\]
Indeed, by Lemma~\ref{lem:smooth} and Corollary~\ref{refine0}, we have 
$[C]\in \Irr_{\CG}(W)$ and so $\chi=[C]$. This yields the left hand side. 
On the other hand, by Proposition~\ref{refine2}, we have $c_{d,\chi}=n_d
\langle [C],\chi\rangle_W=1$. Furthermore, by Lemma~\ref{lem:smooth}, all 
the involutions in $C$ are contained in $\DC$. Hence, $\CC \cap C=
\varnothing$ unless $d\in \CC$, in which case $\CC\cap C= \{d\}$. Thus, 
the right hand side of the $(\CG,C,\CC)$-identity reduces to the expression 
above.
\end{exemple}

\medskip

\begin{coro}\label{coro:conjugues}
{\bfit Recall our assumption that Lusztig's Conjectures P's for 
$(W,S,\ph)$ hold}. Let $\CG$ be a smooth two-sided cell. Then all the 
involutions in $\CG$ are conjugate in $W$. 
\end{coro}

\begin{proof} Let $\CG=C_1 \amalg \ldots \amalg C_n$ be the partition
of $\CG$ into left cells. By Lusztig's Conjecture P13, for each $i$ there 
is a unique $d_i \in \DC \cap C_i$. On the other hand, by 
Lemma~\ref{lem:smooth}, all involutions in $\CG$ are contained in $\DC$.
It follows that $\{d_1,\ldots,d_n\}$ is precisely the set of involutions
in $\CG$. Now let $\CC$ be the conjugacy class containing $d_1$. Then the 
identity in Example~\ref{smooth1} reads:
\[ \sum_{w \in \CC \cap \CG} c_{w,\chi}=\chi(1)n_{d_1} \neq 0\qquad 
\mbox{(since $d_1 \in \CC$)}.\]
Similarly, for any $i\geq 2$, we have 
\[\sum_{w \in \CC \cap \CG} c_{w,\chi}= \left\{\begin{array}{cl} 
\chi(1)n_{d_i} & \quad \mbox{if $d_i \in \CC$},\\
0 & \quad \mbox{otherwise}.\end{array}\right.\]
Since the left hand side is non-zero, we conclude that $d_i \in \CC$,
as claimed.
\end{proof}

\medskip
\begin{exemple} \label{remsn1} Let $W=\SG_n$ be of type $A_{n-1}$ with 
generators given by the basic transpositions $s_i=(i,i+1)$ for $1 \leq 
i \leq n-1$. Then, as already mentioned in Example~\ref{p115eq}, all 
the two-sided cells in $W$ are smooth and so we now recover a known 
result of Sch\"utzenberger~\cite{sch} in this case. An elementary proof 
that Lusztig's Conjectures P's for $(W,S,\ph)$ hold is given in 
\cite{mysn} (see also \cite[\S 2.8]{geja}). We can now also explicitly 
determine the conjugacy class of involutions associated with a two-sided 
cell. Indeed, it is well-known that the irreducible characters of $W=\SG_n$ 
have a natural labelling by the partitions of $n$; we write this in the form
\[ \Irr(\SG_n)=\{\chi^\alpha \mid \alpha \vdash n\}.\]
For example, $\chi^{(n)}$ is the trivial character and $\chi^{(1^n)}$ is 
the sign character. For $\alpha \vdash n$, let $\CG_\alpha$
be the unique two-sided cell such that $\chi^\alpha \in 
\Irr_{\CG_{\alpha}}(\SG_n)$. Since every two-sided cell is smooth, the sets 
$\{\CG_\alpha \mid \alpha \vdash n\}$ are precisely the two-sided cells 
of $\SG_n$. Given $\alpha \vdash n$, let $\CC_\alpha$ be the unique conjugacy 
class of involutions such that $\CC_\alpha\cap\CG_\alpha \neq \varnothing$.
Let $\alpha^*$ denote the transpose partition and $w_{\alpha^*}$ be 
the longest element in the Young subgroup $\SG_{\alpha^*}\subseteq \SG_n$. 
Then it is well-known that 
\[ \big\langle \Ind_{\SG_{\alpha^*}}^{\SG_n}(\varepsilon_{\alpha^*}),
\chi^\alpha\big\rangle_{\SG_n}=1 \qquad \mbox{where} \qquad 
\varepsilon_{\alpha^*}= \mbox{ sign character of $\SG_{\alpha^*}$}.\]
Using the formula for $\ab_{\chi^\alpha}$ in \cite[4.4]{LuB}, one also
sees that $\ab_{\chi^\alpha}=\ell(w_{\alpha^*})$. Hence, by 
Example~\ref{lem:useful}, we have $w_{\alpha^*} \in \CG_\alpha$ and so 
\[\CC_{\alpha}=\mbox{ conjugacy class containing $w_{\alpha^*}$}.\]
The discussion of this example will be continued in Example~\ref{koan}.
\end{exemple}

\medskip

\begin{exemple} \label{remab}
Assume that $(W,S)$ if of type $B_n$, as in Example~\ref{smoothb}. 
Let $b > (n-1) a$. Then the fact that all the involutions contained 
in a two-sided cell are conjugate can be proved directly from the 
combinatorial description given in~\cite[Theorem~7.7]{bonnafe iancu} 
and~\cite[Theorem~3.5~and~Corollary~5.2]{bonnafe two}, by using 
Sch\"utzenberger's result for the symmetric group~\cite{sch}. Also, 
for more general values of $a,b$, a conjectural description of left, right 
and two-sided cells is provided by~\cite[Conjectures~$\Arm^+$~and~B]{BGIL}: 
it would be interesting to see if the conjecture we have stated in the 
introduction is compatible with this conjectural combinatorial construction.
\end{exemple}

\medskip

\begin{rema} \label{f4i2} 
If $W$ is of type $F_4$ or $I_2(m)$ and $\ph$ is a general weight function, 
then Lusztig's Conjectures P's for $(W,S,\ph)$ are known to hold; see 
\cite[\S 5]{myp115}. In these cases, using the explicit knowledge of 
the cells and the classes of involutions (see \cite{geck f4} for type 
$F_4$ and \cite[\S 8]{lusztig} for type $I_2(m)$), one can directly check
that, if $C$ and $C'$ are two left cells contained in the same two-sided 
cell, then $\CC_2(C)=\CC_2(C')$. This provides some support for the 
general conjecture stated in the introduction.
\end{rema}

%W=coxeter("F",4)
%def cedric(W,param,v):
%  def c2(W,el,l):
%    l0=[W.wordtoperm(el[x]) for x in l]
%    l1=[]
%    for x in l0:
%      if W.permorder(x)<=2 and not x in l1:
%        l1.extend(conjugacyclass(W,x))
%    l1=[el.index(W.permtoword(x)) for x in l1]
%    l1.sort()
%    return l1
%  k=klpolynomials(W,param,v)
%  a=[[c2(W,k['elms'],l) for l in k['lcells'] 
%              if l[0] in t] for t in k['tcells']]
%  print([len(i) for i in a])
%  return [len(noduplicates(x)) for x in a]

%Furthermore, Conjecture~\ref{cspec} also holds; 
%see \cite[\S 5]{geck pycox}. Hence, we can conclude that 
%
%Of course, in these cases, one can also directly check this assertion
%from the explicit knowledge of the cells and the classes of involutions;
%see \cite{geck f4} (type $F_4$) and \cite[\S 8]{lusztig} (type $I_2(m)$).

\section{The equal parameter case}\label{section:eqpa}

\springer{{\bf Hypothesis.} {\it From now until the end of this paper, we 
assume that we are in equal parameter case where $\Gamma=\ZM$ and 
$\varphi(s)=1$ for all $s \in S$.}}

\medskip
Under this hypothesis, as already mentioned in Example~\ref{p115eq}, it 
is known that Lusztig's Conjectures P's for $(W,S,\ph)$ hold. One further 
distinctive feature of this case is the existence of {\it special} characters.
For $\chi \in \Irr(W)$, let $\bb_\chi$ denote the smallest 
integer $i \geq 0$ such that $\chi$ occurs in the $i$th symmetric 
power of the standard reflection representation of $W$. Then, following
Lusztig \cite[4.1]{LuB}, $\chi$ is called {\it special} if $\ab_\chi=
\bb_\chi$.  Let 
\[\SC(W):=\{ \chi \in \Irr(W)\; \mid \; \ab_\chi=\bb_\chi\}\]
be the set of special characters of $W$. It is known that 
\begin{equation*}
|\SC(W)\cap \Irr_\CG(W)|=1 \qquad \mbox{for every two-sided cell
$\GC$ of $W$}.  \tag{$\diamondsuit_1$}
\end{equation*}
This is seen as follows. Consider the partition of $\Irr(W)$ in 
terms of ``families'', as defined in \cite[4.2]{LuB}. (The same 
definition also works for groups of non-crystallographic type; see
\cite[\S 6.5]{gepf}.) By \cite[4.14]{LuB}, every such family 
contains a unique special character (and this also holds for 
non-crystallographic types; see \cite[\S 6.5]{gepf}). Hence,
($\diamondsuit_1$) follows from the known fact that the partition of 
$\Irr(W)$ into families coincides with the partition in
Definition~\ref{def:fam}. For Weyl groups, this appeared in 
\cite[Theorem~5.25]{LuB}. A different argument based on certain 
``positivity'' properties of the Kazhdan--Lusztig basis is given in 
\cite[Prop.~23.3]{lusztig}; the same argument also works for the 
non-crystallographic types, where the analogous ``positivity'' 
properties are known by explicit computation; see Alvis \cite{Al}, 
DuCloux \cite{fokko}.

Now let $\CG$ be a two-sided cell. Then, if $\chi$ denotes the 
unique character in $\SC(W)\cap \Irr_\CG(W)$, we have
\begin{equation*}
(-1)^{\ab_{\chi}+\ell(w)}c_{w,\chi}>0 \quad \mbox{for all 
$w\in C\cap C^{-1}$}, \tag{$\diamondsuit_2$}
\end{equation*}
where $C$ is any left cell contained in $\CG$.
This holds by \cite[Prop.~3.14]{lulc} for Weyl groups and by 
\cite[Rem.~5.12]{geck pycox} for the remaining types. Note that, in the 
notation of \cite[\S 3]{lulc}, the special character $\chi$ corresponds 
to the pair $(1,1)\in \MC(G_{\CG})$ where $G_\CG$ is the finite group 
associated with $\CG$ (see also \cite[4.14.2]{LuB}). The factor 
$(-1)^{\ab_\chi+\ell(w)}$ comes from the identity \cite[3.5(a)]{lulc} which 
relates the leading coefficients to the characters of Lusztig's
asymptotic algebra $J$.

\begin{prop} \label{eqpa1} Recall our assumption that we are in the equal 
parameter case. Let $\CG$ be a two-sided cell and $C,C'$ be left cells 
of $W$ which are contained in $\CG$. Let $\CC$ be a conjugacy class of 
involutions in $W$. Then $\CC\cap C \neq \varnothing$ if and only if 
$\CC\cap C'\neq \varnothing$. In particular, we have $\CC_2(C)=\CC_2(C')$.
\end{prop}

\begin{proof} We consider the $(\CG,C,\CC)$-identity in Lemma~\ref{main1} 
with respect to the unique special character $\chi \in \SC(W) \cap 
\Irr_{\CG}(W)$. Since the sign character of $W$ is constant on $\CC$, we
can write this identity in the form: 
\[ \langle [C],\chi\rangle_W\sum_{w \in \CC \cap \CG} (-1)^{\ell(w)}
c_{w,\chi}=\chi(1) \sum_{w \in \CC \cap C} (-1)^{\ell(w)}
c_{w,\chi} .\]
Multiplying both sides by $(-1)^{\ab_{\chi}}$, we obtain:
\[ \langle [C],\chi\rangle_W\sum_{w \in \CC \cap \CG} 
(-1)^{\ab_{\chi}+\ell(w)} c_{w,\chi}=\chi(1) 
\sum_{w \in \CC \cap C} (-1)^{\ab_{\chi}+\ell(w)}c_{w, \chi}.\]
By ($\diamondsuit_2$), we have $c_{d,\chi}\neq 0$ and so 
$\langle [C], \chi\rangle_W \neq 0$; see Proposition~\ref{refine2}. 
Thus, we obtain  
\[ \sum_{w \in \CC \cap \CG} (-1)^{\ab_{\chi}+\ell(w)} 
c_{w,\chi}= \frac{\chi(1)}{\langle [C],\chi \rangle_W} \sum_{w \in \CC 
\cap C} (-1)^{\ab_{\chi}+\ell(w)} c_{w,\chi}.\]
Let us denote by $\Upsilon(\CC,C)$ the expression on the right hand
side of this identity. Since the left hand side does not depend on 
$C$, we have $\Upsilon(\CC,C)=\Upsilon(\CC,C')$. Consequently, we have 
\[\sum_{w \in \CC \cap C} (-1)^{\ab_{\chi_{\cG}}+\ell(w)}c_{w,\chi}
\neq 0 \qquad \Longleftrightarrow \qquad \sum_{w \in \CC \cap C'} 
(-1)^{\ab_{\chi_{\cG}}+\ell(w)}c_{w,\chi} \neq 0.\]
Finally, by ($\diamondsuit_2$), we have 
\[(-1)^{\ab_{\chi}+\ell(w)} c_{w,\chi}>0 \qquad \mbox{for 
all $w \in \CC \cap C$ and for all $w \in \CC \cap C'$}.\]
Thus, the left hand side of the above equivalence is non-zero 
if and only if $\CC \cap C\neq \varnothing$, and, similarly, the right
left hand side is non-zero if and only if $\CC \cap C'\neq \varnothing$. 
\end{proof}

\medskip

\begin{defi} \label{exchar} A character $\chi \in \Irr(W)$ is 
called {\it exceptional} if there exists some $w \in W$ such that
$c_{w,\chi}\neq 0$ and $\ab_\chi\not\equiv \ell(w) \bmod 2$.
\end{defi}

\medskip

\begin{rema} \label{remeq} One easily checks that there is a well-defined 
ring homomorphism $\alpha \colon \HC \rightarrow \HC$ such that $\alpha(v)
=-v$ and $\alpha(r)=r$ for all $r \in \RM$ and $\alpha(T_w)=(-1)^{\ell(w)}T_w$ 
for all $w \in W$. (See Lusztig \cite[3.2]{Lu2}.) Now, for $\chi \in 
\Irr(W)$, we have $\chi_\ph(T_w) \in {\RM}[v,v^{-1}]$ for all $w \in W$. 
Composing the action of $\HC$ on a representation affording $\chi_\ph$ 
with $\alpha$, we see that there is a well-defined $\tilde{\chi}\in 
\Irr(W)$ such that 
\[ \tilde{\chi}_\ph(T_w)=(-1)^{\ell(w)}\chi_\ph(T_w)\,\big|_{v\mapsto -v} 
\qquad \mbox{for all $w \in W$}.\]
By the definition of $\ab_\chi$ and $c_{w,\chi}$, this implies that
\[ \ab_{\tilde{\chi}}=\ab_{\chi} \qquad\mbox{and} \qquad 
c_{w,\tilde{\chi}}=c_{w,\chi}^*\quad \mbox{for all $w \in W$}.\]
Thus, $\chi$ is exceptional if and only if $\chi \neq \tilde{\chi}$.
Using Corollary~\ref{refine0} we see that, for a two-sided cell $\CG$ 
of $W$, we have 
\[ \chi \in \Irr_{\GC}(W) \qquad \Longleftrightarrow \qquad \tilde{\chi}
\in \Irr_{\CG}(W).\]
Note that there do exist cases for which $\chi \neq \tilde{\chi}$.
For example, let $(W,S)$ be of type $E_7$. Then, by \cite[5.22.2]{LuB},
there exists an involution $x \in W$ such that $c_{x,\chi}\neq 0$ 
and $\ab_\chi\not\equiv \ell(x) \bmod 2$ for the special character denoted 
$\chi=512_a'$. In type $E_8$, examples are given by the special characters 
$4096_z$ and $4096_x'$; see \cite[5.23.2]{LuB}.
\end{rema}

\medskip

\begin{exemple} \label{remeq1} Assume that $(W,S)$ is irreducible. By the 
previous remark we see that, if $v^{\ell(w)}\chi(T_w) \in \RM[v^2]$ for all 
$w \in W$, then $\chi$ is non-exceptional. So, by \cite[Example~9.3.4]{gepf},
all $\chi\in \Irr(W)$ are non-exceptional unless $(W,S)$ is of type $H_3$, 
$H_4$, $E_7$, $E_8$ and $\chi$ is one of the characters listed in 
\cite[Example~9.2.3]{gepf}. (This list includes the characters $512_a'$,
$4096_z$, $4096_x'$ already mentioned in Remark~\ref{remeq}.) The degree 
of such an exceptional character is a power of~$2$; furthermore, we have 
$v^{\ell(w_0)}\chi(T_{w_0}) \not\in \RM[v^2]$ where $w_0 \in W$ is the 
longest element. 

In particular, if $(W,S)$ is of classical type, then all $\chi \in \Irr(W)$ 
are non-exceptional.
\end{exemple}

\medskip

\begin{exemple} \label{class1} Assume that $(W,S)$ is irreducible and of
classical type. (Also recall that we are in the equal parameter case). Let 
$\CG,C,\CC$ be as in Lemma~\ref{main1}. Let $\chi \in \SC(W)$ be the unique 
special character in $\Irr_{\CG}(W)$. Then we claim that 
\[ |\CC \cap \CG|=\chi(1) |\CC \cap C|.\]
This is seen as follows. As already noted in the proof of 
Proposition~\ref{eqpa1}, we have $\langle [C],\chi\rangle_W \neq 0$. 
By \cite[12.13]{LuB}, every left cell module for $W$ is 
multiplicity--free and so $\langle [C],\chi \rangle_W =1$. Consequently,
the $(\CG,C,\CC)$-identity in Lemma~\ref{main1} reduces to: 
\[ \sum_{w \in \CC \cap \CG} c_{w,\chi}= \chi(1)\sum_{w \in \CC \cap C} 
c_{w,\chi}.\]
So it remains to show that 
\[ c_{w,\chi}=c_{w,\chi}^*=1 \qquad \mbox{for all $w \in \CC \cap \CG$}.\] 
Now, the first equality holds since $\chi$ is non-exceptional; see 
Example~\ref{remeq1}. Furthermore, by \cite[3.10(b)]{lulc}, we have 
$c_{w,\chi}\in \{0,\pm 1\}$ for all $w \in W$. Hence, the second equality 
immediately follows from ($\diamondsuit_2$).

In particular, the equality $|\CC \cap \CG|=\chi(1) |\CC \cap C|$ 
shows that the cardinality $|\CC \cap C|$ does not depend on $C$. This 
phenomenon is related to a conjecture of Kottwitz \cite{kottwitz}, which
we shall now explain.
\end{exemple}

\medskip
\begin{defi}[\protect{\cite{kottwitz}, \cite{luvo}, \cite{luvo1}}] 
\label{defko} Let $\CC$ be a union of conjugacy classes 
of involutions in $W$. Let $V_{\CC}$ be an $\RM$-vector space with a basis
$\{a_w \mid w \in \CC\}$. Then, by \cite[6.3]{luvo} and \cite{luvo1}, there 
is a linear action of $W$ on $V_{\CC}$ such that, for any $s \in S$ and 
$w \in \CC$, we have 
\[ s.a_w=\left\{\begin{array}{cl} -a_w & \qquad \mbox{if $sw=ws$ and
$\ell(sw)<\ell(w)$},\\ a_{sws} & \qquad \mbox{otherwise}.\end{array}\right.\]
Let $\rho_{\CC}$ denote the character of this representation of $W$ on 
$V_{\CC}$. 
\end{defi}

\medskip

\begin{conjecture}[Kottwitz \protect{\cite[\S 1]{kottwitz}}] \label{coko}
Let $\CC$ be a union of conjugacy classes of involutions and $C$ be a 
left cell of $W$. Then $\langle \rho_{\CC}, [C]\rangle_W=|\CC \cap C|$.
\end{conjecture}

\medskip

\begin{rema} \label{remako} The fact that $\rho_{\CC}$ indeed is equal 
to the character originally constructed in \cite{kottwitz} is shown in 
\cite[Rem.~2.2]{gema}. Note also that, if $\CC=\CC_1 \amalg \ldots 
\amalg \CC_r$ is the partition of $\CC$ into conjugacy classes, then we 
certainly have 
\[ \rho_{\CC}=\rho_{\CC_1}+\cdots +\rho_{\CC_r}.\]
Hence, it is sufficient to prove the above conjecture for the case
where $\CC$ is a single conjugacy class of involutions.
\end{rema}

A strong support is provided by the following general result.

\medskip

\begin{theo}[Marberg \protect{\cite[1.7]{marberg}}]  \label{coko1} Let
$\Ib$ denote the set of all involutions in $W$. Then $\langle \rho_{\Ib}, 
[C]\rangle_W=|\Ib \cap C|$ for every left cell $C$ in $W$.
\end{theo}

\medskip
Already Kottwitz \cite{kottwitz} showed that his conjecture holds in 
type $A_{n-1}$; see Example~\ref{koan} below. The aim of the following
three sections is to deal with types $B_n$ and $D_n$; see Theorems~\ref{kottB} 
and~\ref{kottD}. This will rely in an essential way on the above 
identity in Example~\ref{class1}. As far as the exceptional types are 
concerned, Casselman \cite{cass} has verified the conjecture by explicit 
computation for $F_4$ and $E_6$; in \cite{geck pycox}, this is extended to 
$E_7$. Marberg \cite{marberg} verified the conjecture for the 
non-crystallographic types $H_3$, $H_4$, $I_2(m)$. Thus, the only remaining 
case is type $E_8$, which is currently being considered by A. Halls at the 
University of Aberdeen.

\medskip

\begin{exemple} \label{koan} Let $W=\SG_n$ be of type $A_{n-1}$ with 
generators given by the basic transpositions $s_i=(i,i+1)$ for $1 \leq 
i \leq n-1$. A complete set of representatives of the conjugacy classes 
of involutions is given by the elements
\[\sigma_j:=s_1s_3\cdots s_{2j-1}\in \SG_n \qquad \mbox{where} \qquad
0 \leq 2j\leq n.\]
(Thus, $\sigma_j$ is the product of $j$ disjoint transpositions and
$\sigma_j$ has precisely $n-2j$ fixed points on $\{1,\ldots,n\}$.) Note
that $\sigma_{j}$ has minimal length in its conjugacy class; see 
\cite[3.1.16]{gepf}. Let $\CC_{j}$ be the conjugacy class containing 
$\sigma_{j}$ and write $\rho_{j}= \rho_{\CC_{j}}$. As in
Example~\ref{remsn1}, we write $\Irr(\SG_n)=\{\chi^{\alpha} 
\mid \alpha \vdash n\}$.  Then, by \cite[3.1]{kottwitz}, we have
\begin{equation*}
\langle \rho_{j},\chi^{\alpha}\rangle_{\SG_n}=\delta_{n-2j=t},\tag{a}
\end{equation*}
where $t$ is the number of odd parts of the conjugate partition $\alpha^*$.
In particular, if $\Ib$ denotes the set of all involutions in $\SG_n$, then
\[ \rho_{\Ib}=\sum_{j} \rho_{j}=\sum_{\alpha \vdash n} \chi^{\alpha}.\]
For later reference, we explicitly note the following special case of (a).
Let $\alpha=(1^n)$; then $\chi^\alpha=\varepsilon$ is the sign character 
of $\SG_n$. Then (a) yields:
\begin{equation*}
\langle \rho_{j},\varepsilon\rangle_{\SG_n}=\left\{\begin{array}{cl}
1 & \quad \mbox{if $j=\lfloor n/2 \rfloor$}, \\ 0 & \quad \mbox{otherwise}.
\end{array}\right.\tag{b}
\end{equation*}
We have now all ingredients in place to verify that Kottwitz' Conjecture 
holds. Indeed, first note that the longest element in $\SG_n$ has precisely 
one fixed point on $\{1, \ldots, n\}$ if $n$ is odd, and no fixed point at 
all if $n$ is even. Now let $\alpha \vdash n$ and let $w_{\alpha^*}$ be the 
longest element in the Young subgroup $\SG_{\alpha^*}\subseteq \SG_n$, as 
in Example~\ref{remsn1}.  If $\alpha_1^*,\ldots,\alpha_r^*$ are the 
non-zero parts of $\alpha^*$, then $\SG_{\alpha^*} \cong \SG_{\alpha_1^*}
\times\ldots \times \SG_{\alpha_r^*}$ and so the number of fixed points of 
$w_{\alpha^*}$ on $\{1,\ldots,n\}$ is the number $t$ of odd parts of 
$\alpha^*$. Thus, $w_{\alpha^*}$ is conjugate to $\sigma_{(n-t)/2}$ and so 
we can reformulate (a) as follows. Let $\CC$ be a conjugacy class of 
involutions in $\SG_n$. Then 
\begin{equation*}
\langle \rho_{\CC},\chi^{\alpha}\rangle_{\SG_n}=\left\{\begin{array}{cl}
1 & \qquad \mbox{if $w_{\alpha^*} \in \CC$},\\ 0 & \qquad \mbox{otherwise}.
\end{array}\right.\tag{c}
\end{equation*}
Comparison with Example~\ref{remsn1} now shows that 
Conjecture~\ref{coko} holds in this case.
\end{exemple}

\section{An inductive approach to Kottwitz' Conjecture}
\label{section:kottred}

\medskip
We keep the basic assumptions of the previous section. The results in 
this section will provide some ingredients for an inductive proof of
Kottwitz' Conjecture~\ref{coko}. 

\medskip

Let $\CG$ be a two-sided cell of $W$. We shall say that ``{\it Kottwitz'
Conjecture holds for $\CG$}'' if, for any conjugacy class of involutions
$\CC$ in $W$, we have
\[ \langle \rho_{\CC},[C]\rangle_W=|\CC \cap C| \qquad \mbox{for all
left cells $C \subseteq \CG$}.\]

\medskip
\begin{rema} \label{remw0} Let $w_0 \in W$ be the longest element.
Let $C$ be a left cell of $W$. Then, by \cite[5.14]{LuB}, the set
$C w_0$ also is a left cell and we have 
\[ [Cw_0] =[C] \otimes \varepsilon \qquad \mbox{where} \qquad 
\mbox{$\varepsilon=$ sign character of $W$}.\]
Now let $\CG$ be a two-sided cell. Then $\CG w_0$ also is a two-sided cell
and we have
\[ \Irr_{\CG w_0}(W)=\Irr_{\CG}(W) \otimes \varepsilon:=\{ \chi \otimes
\varepsilon \mid \chi \in \Irr_{\CG}(W)\}.\]
\end{rema}

\medskip
\begin{lem} \label{epsrho} Assume that the longest element $w_0 \in W$
is central in $W$. Let $\CC$ be a union of conjugacy classes of 
involutions in $W$. Then $\CC w_0$ also is a union of conjugacy classes 
of involutions and we have $\rho_{\CC w_0}=\rho_{\CC} \otimes \varepsilon$.
\end{lem}

\begin{proof} It is sufficient to prove this in the case where
$\CC$ is a single conjugacy class. Let $l_0:=\min\{\ell(w) \mid w \in \CC\}$.
Then $\ell(w)-l_0$ is even for every $w \in \CC$. So, for any $w \in \CC$, 
there is a well-defined integer $m(w)$ such that $\ell(w)-l_0=2m(w)$. Now
we perform a change of basis in $V_{\CC}$: we set $a_w':=(-1)^{m(w)}a_w$
for $w \in \CC$. Then the action of $W$ on $V_{\CC}$ is given by the
following formulae, where $s \in S$ and $w \in \CC$:
\[ s.a_w'=\left\{\begin{array}{cl} -a_w' & \qquad \mbox{if $sw=ws$ and
$\ell(sw)<\ell(w)$},\\ a_w' & \qquad \mbox{if $sw=ws$ and $\ell(sw)>\ell(w)$},\\
-a_{sws}' & \qquad \mbox{otherwise (that is, if $sw\neq ws$)}.\end{array}
\right.\]
Note that, since $w \in \CC$ is an involution, we have $\ell(sw)>\ell(w)$
if and only if $\ell(ws)>\ell(w)$; hence, if $sw\neq ws$, then $\ell(sws)=\ell(w)\pm 2$ 
(see \cite[1.2.6]{gepf}) and so $a_{sws}'=-a_w'$. Furthermore, it is
well-kwown that $\ell(yw_0)=\ell(w_0)-\ell(y)$ for every $y \in W$. Hence, we
can also write the above formulae in the following form:
\[ s.a_w'=\left\{\begin{array}{cl} a_w' & \qquad \mbox{if $sww_0=ww_0s$ and
$\ell(sww_0)<\ell(ww_0)$},\\ -a_w' & \qquad \mbox{if $sww_0=ww_0s$ and
$\ell(sww_0)>\ell(ww_0)$},\\ -a_w' & \qquad \mbox{otherwise}.\end{array} \right.\]
Tensoring with $\varepsilon$, we see that we obtain exactly the same
formulae as for the action of $W$ on $V_{\CC w_0}$.
\end{proof}

\medskip
\begin{lem} \label{epsrho1} Assume that the longest element $w_0 \in W$
is central in $W$. Let $\CG$ be a two-sided cell. Then Kottwitz'
Conejcture holds for $\CG$ if and only if Kottwitz' Conjecture holds
for $\CG w_0$.
\end{lem}

\begin{proof} Assume that Kottwitz' Conjecture holds for $\CG$. Let 
$\CC$ be a conjugacy class of involutions in $W$. Let $C$ be a left 
cell contained in $\CG w_0$. Then $C w_0$ is a left contained in $\CG$
and we obtain
\[ \langle \rho_{\CC},[C]\rangle_W=\langle \rho_{\CC} \otimes \varepsilon,
[C]\otimes \varepsilon\rangle_W=\langle \rho_{\CC w_0}, [C w_0]\rangle_W\]
where the last equality holds by Remark~\ref{remw0} and Lemma~\ref{epsrho}.
Now, by assumption, the right hand side equals $|(\CC w_0) \cap (C w_0)|=
|\CC \cap C|$, as desired. The reverse implication is then clear. 
\end{proof}

\medskip

\begin{defi}[\protect{\cite[8.1]{LuB}}] \label{cuspfam} Let $\CG$ be
a two-sided cell in $W$. We say that $\CG$ is {\em strongly non-cuspidal} 
if there exists a proper standard parabolic subgroup $W' \subsetneqq W$ and
a two-sided cell $\CG'$ in $W'$ such that the ``truncated induction''
$\Jrm_{W'}^W$ (as defined in \cite[4.1.7]{LuB}) establishes a 
bijection 
\[ \Irr_{\CG'}(W') \rightarrow \Irr_{\CG}(W), \qquad \chi' \mapsto
\Jrm_{W'}^W(\chi').\]
We say that $\CG$ is {\em non-cuspidal} if $\CG$ or $\CG w_0$ is strongly 
cuspidal (where $w_0 \in W$ is the longest element). Finally, we say 
that $\CG$ is {\em cuspidal} if $\CG$ is not non-cuspidal. 

(Note that, in \cite[8.1]{LuB}, the formulation is in terms of 
``families'' of $\Irr(W)$; however, as already mentioned at the 
beginning of Section~\ref{section:eqpa}, it is known that the sets 
$\Irr_{\CG}(W)$ are precisely the ``families'' of $\Irr(W)$.)
\end{defi}

\medskip

\begin{rema} \label{cusp1} Let $\CG$ be a two-sided cell in $W$ and 
assume that $\CG$ is {\em strongly non-cuspidal}. Let $W',\CG'$ be as
in Definition~\ref{cuspfam}.  Let 
\[ \chi=\Jrm_{W'}^W(\chi') \in \Irr_{\CG}(W) \qquad \mbox{where}
\qquad \chi' \in \Irr_{\CG'}(W').\]
By the definition of the truncated induction, we have $\ab_{\chi}=
\ab_{\chi'}$. Using \cite[4.1.6]{LuB}, one easily sees that also 
$f_{\chi'}=f_{\chi'}$.  In particular, $\CG$ is smooth if and only if
$\CG'$ is smooth (see Lemma~\ref{lem:smooth}).
\end{rema}

\medskip

\begin{lem} \label{kott0} Let $\CG$ be a strongly non-cuspidal two-sided 
cell in $W$. Let $W',\CG'$ be as in Definition~\ref{cuspfam}. Then
Kottwitz's Conjecture holds for $\CG$ if the following three conditions 
are satisfied.
\begin{itemize}
\item[(K1)] For any conjugacy class of involutions $\CC$ in $W$ and any 
left cells $C_1,C_2\subseteq \CG$ such that $[C_1]=[C_2]$, we have 
$|\CC \cap C_1|= |\CC\cap C_2|$.
\item[(K2)] Kottwitz's Conjecture holds for the two-sided cell $\CG'$ in $W'$.
\item[(K3)] For any conjugacy class of involutions $\CC$ in $W$ such that
$\CC \cap W' \neq \varnothing$, we have 
\[\langle \rho_{\CC\cap W'},\chi'\rangle_{W'}\leq \big\langle \rho_{\CC},
\Jrm_{W'}^W (\chi') \big\rangle_{W} \qquad \mbox{for all $\chi'\in 
\Irr_{\CG'}(W')$}.\]
 \end{itemize}
\end{lem}

\begin{proof} Let $\CC$ be any conjugacy class of involutions in $W$. First
we show that 
\begin{equation*}
\langle \rho_{\CC},[C]\rangle_W \geq |\CC \cap C| \qquad \mbox{for all
left cells $C \subseteq \CG$}.\tag{$*$}
\end{equation*}
Indeed, let $C$ be a left cell of $W$ which is contained in $\CG$. 
If $\CC \cap  C=\varnothing$, then ($*$) is obvious. Now assume that  
$\CC \cap  C\neq \varnothing$. By \cite[\S 3]{lu3} (see also 
\cite[Lemma~5.6]{geck plus}), there exists a left cell $C'$ of $W'$
which is contained in $\CG'$ and such that $[C]=\Jrm_{W'}^W([C'])$. 
So we have
\[ \langle \rho_{\CC},[C]\rangle_W=\langle \rho_{\CC},\Jrm_{W'}^W([C'])
\rangle_W.\]
Using now (K2) and (K3), we obtain 
\[ \langle \rho_{\CC},\Jrm_{W'}^W([C'])\rangle_W\geq
\langle \rho_{\CC \cap W'},[C']\rangle_{W'}=|(\CC \cap W')\cap C'|=
|\CC \cap C'|.\] 
On the other hand, let $C_1$ be the left cell of $W$ such that $C'
\subseteq C_1$. Then we also have $[C_1]=\Jrm_{W'}^W([C'])$; see 
\cite[5.28]{LuB} (or the argument in the proof of Case~1 in 
\cite[Lemma~22.2]{lusztig}). 
Now, since $[C_1]=\Jrm_{W'}^W([C'])$ and since $\Jrm_{W'}^W$ establishes
a bijection between $\Irr_{\CG'}(W)$ and $\Irr_{\CG}(W)$, we conclude that
$[C']$ and $[C_1]$ have the same number of irreducible constituents 
(counting multiplicities). Hence, by Proposition~\ref{myinv}, $C'$ and 
$C_1$ contain the same number of involutions. Consequently, since 
$C' \subseteq C_1$, all the involutions in $C_1$ must be contained in $C'$ 
and so $\CC \cap C'=\CC \cap C_1$. In particular, this shows that 
$|\CC \cap C'| =|\CC\cap C_1|=|\CC\cap C|$, where the second equality holds 
by (K1).  Thus, ($*$) is proved. Once this is established, it actually 
follows that we must have equality in ($*$). Indeed, let $\CC_1,\ldots,
\CC_m$ be the conjugacy classes of involutions in $W$; then $\Ib=\CC_1 \cup 
\ldots \cup \CC_m$ is the set of all involutions in $W$. By ($*$), we have
\[ \langle \rho_{\Ib},[C]\rangle_W=\sum_{1\leq i \leq m}
\langle \rho_{\CC_i},[C]\rangle_W\geq \sum_{1\leq i \leq m} |\CC_i \cap C|
=|\Ib\cap C|.\]
However, by Theorem~\ref{coko1}, we know that the left hand side equals the
right hand side. Hence, all the inequalities in ($*$) must be equalities, 
as claimed. Thus, Kottwitz's Conjecture holds for $\CG$. 
\end{proof}

\medskip
\begin{rema} \label{indko} We note that an analogous version of the 
inequality in (K3) always holds where $\Jrm_{W'}^W(\chi')$ is replaced 
by $\Ind_{W'}^W(\chi')$. In fact, for any parabolic subgroup $W' 
\subseteq W$ and any conjugacy of involutions $\CC$ in $W$ such that 
$\CC' \cap W' \neq \varnothing$, we have 
\[\langle \rho_{\CC\cap W'},\chi'\rangle_{W'}\leq \big\langle \rho_{\CC},
\Ind_{W'}^W (\chi') \big\rangle_{W} \qquad \mbox{for all $\chi'\in 
\Irr(W')$}.\]
This is seen as follows. Let $V_{\CC}$ be as in Definition~\ref{defko}. From
the formulae for the action of $W$ on $V_{\CC}$, it is clear that the
subspace $U\subseteq V_{\CC}$ spanned by the basis elements $\{a_w \mid
w \in \CC \cap W'\}$ is a $W'$-submodule. Furthermore, the character 
of this $W'$-module is just $\rho_{\CC\cap W'}$. Thus, we can write 
$\Res_{W'}^W(\rho_{\CC})=\rho_{\CC\cap W'}+\psi$ for some character $\psi$
of $W'$. This yields that 
\[\langle\rho_{\CC'\cap W'},\chi'\rangle_{W'} \leq \langle
\Res_{W'}^W(\rho_{\CC}),\chi'\rangle_{W'}\qquad \mbox{for all $\chi'\in 
\Irr(W')$}\]
and so the assertion immediately follows by Frobenius reciprocity.
\end{rema}

\medskip

\begin{lem}  \label{indko1} Let $W'\subseteq W$ be a standard parabolic
subgroup and $\CC'$ be a conjugacy class of involutions in $W'$. Let
$\CC$ be the conjugacy class of $W$ such that $\CC'\subseteq \CC$. Then
\[ \big\langle \rho_{\CC},\chi \big\rangle_W\leq \big\langle\Ind_{W'}^W
(\rho_{\CC'}), \chi\big\rangle_W \qquad \mbox{for all $\chi \in \Irr(W)$}.\]
\end{lem}

\begin{proof} We can find a representative $\sigma \in \CC'$ such that 
$\sigma$ is the longest element in a standard parabolic subgroup 
$W''\subseteq W'$ and such that $\sigma$ is central in $W''$; see 
\cite[3.2.10]{gepf}. Then, by Kottwitz' original construction in 
\cite{kottwitz}, we have 
\[ \rho_{\CC}=\Ind_{C_{W}(\sigma)}^W(\varepsilon_\sigma) \qquad 
\mbox{and}\qquad \rho_{\CC'}=\Ind_{C_{W'}(\sigma)}^{W'}
(\varepsilon_\sigma'), \]
where $\varepsilon_\sigma\colon C_W(\sigma) \rightarrow \{\pm 1\}$ and 
$\varepsilon_\sigma'\colon C_{w'}(\sigma) \rightarrow \{\pm 1\}$ are 
certain linear characters. To describe these characters explicitly, let 
$\Phi$ be the root system of $W$; let $\Phi=\Phi^+\amalg \Phi^-$ be the 
decomposition into positive and negative roots (defined by the given set
of generators $S$ of $W$). Let $\Phi''$ be the parabolic subsystem defined 
by $W''$. Then, for any $w \in C_W(\sigma)$, we have $\varepsilon_\sigma(w)=
(-1)^k$ where $k$ is the number of positive roots in $\Phi''$ which are sent
to negative roots by $w$ (see also \cite[Rem.~2.2]{gema}). The definition
of $\varepsilon_\sigma'$ is analogous. By this description, it is clear 
that $\varepsilon_\sigma'$ is the restriction of $\varepsilon_\sigma$ to
$W'$. Hence, we can write 
\[ \Ind_{C_{W'}(\sigma)}^{C_W(\sigma)}(\varepsilon_\sigma')=
\varepsilon_\sigma+\psi\]
for some character $\psi$ of $C_{W}(\sigma)$. By the transitivity of
induction, this yields 
\[ \Ind_{W'}^W(\rho_{\CC'})=\Ind_{C_{W'}(\sigma)}^W(\varepsilon_\sigma')=
\Ind_{C_W(\sigma)}^W \bigl(\varepsilon_\sigma+\psi\bigr)=
\rho_{\CC}+\Ind_{C_W(\sigma)}^W(\psi),\]
which immediately implies the assertion.
\end{proof}

%It remains to consider the case where $\Jrm_{W'}^W$ establishes a bijection 
%\[ \Irr_{\CG'}(W') \rightarrow \Irr_{\CG}(W) \otimes \varepsilon, \qquad 
%\chi' \mapsto \Jrm_{W'}^W(\chi) \otimes \varepsilon.\]
%As already noted in Definition~\ref{cuspfam}, the set $\CG_1:=\CG w_0$ is 
%a two-sided cell and 
%\[\Irr_{\CG_1}(W)=\Irr_{\CG}(W) \otimes \varepsilon;\]
%establishes a bijection 
%\[ \Irr_{\CG'}(W') \rightarrow \Irr_{\CG_1}(W), \qquad \chi' \mapsto
%\Jrm_{W'}^W(\chi').\]
%Now, since $[C w_0]=[C] \otimes \varepsilon$ for any left cell 
%$C \subseteq \CG$ (see once more \cite[5.14]{LuB}), it is clear
%that condition (a) also holds for $\CG_1$. Hence, we can apply the 
%first part of the proof to $\CG_1,\CG'$. Consequently, Kottwitz's 
%Conjecture holds for all left cells of $W$ which are contained in
%$\CG_1$. Using Lemma~\ref{epsrho}, it also follows that Kottwitz's 
%Conjecture also holds for all left cells of $W$ which are contained
%in $\CG$.

\section{Kottwitz' Conjecture for type $B_n$}\label{section:kottBn}

Throughout this section, let $W=W_n$ be of type $B_n$ with
generators $t,s_1,\ldots,s_{n-1}$ and diagram given as follows.
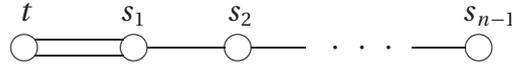
\begin{center}
\begin{picture}(220,30)
\put( 40, 10){\circle{10}}
\put( 44,  7){\line(1,0){33}}
\put( 44, 13){\line(1,0){33}}
\put( 81, 10){\circle{10}}
\put( 86, 10){\line(1,0){29}}
\put(120, 10){\circle{10}}
\put(125, 10){\line(1,0){20}}
\put(155,  7){$\cdot$}
\put(165,  7){$\cdot$}
\put(175,  7){$\cdot$}
\put(185, 10){\line(1,0){20}}
\put(210, 10){\circle{10}}
\put( 38, 20){$t$}
\put( 76, 20){$s_1$}
\put(116, 20){$s_2$}
\put(204, 20){$s_{n{-}1}$}
\end{picture}
\end{center}
The aim of this section is to prove that Conjecture~\ref{coko} holds 
for $W_n$. For this purpose, we first need to recall some results from 
\cite{kottwitz} concerning the decomposition of the character 
$\rho_{\CC}$ into irreducibles. 

\medskip

\begin{exemple} \label{kobn} A complete set of representatives of the 
conjugacy classes of involutions in $W_n$ is given as follows. Let 
$l,j$ be non-negative integers such that $l+2j\leq n$. Then set
\[ \sigma_{l,j}:=t_1\cdots t_ls_{l+1}s_{l+3} \cdots s_{l+2j-1}\in W_n,\]
where $t_1:=t$ and $t_i:=s_{i-1}t_{i-1}s_{s-i}$ for $2 \leq i \leq n$. 
Note that $\sigma_{l,j}$ is the longest element in a parabolic subgroup
of $W_n$ of type $B_l \times A_1 \times \ldots \times A_1$, where the
$A_1$ factor is repeated $j$ times. In particular, $\sigma_{l,j}$ has
minimal length in its conjugacy class; see also \cite[3.2.10]{gepf}. Let 
$\CC_{l,j}$ be the conjugacy class containing $\sigma_{l,j}$ and write 
$\rho_{l,j}= \rho_{\CC_{l,j}}$. The irreducible characters of $W_n$ are 
parametrised by pairs of partitions $(\alpha,\beta)$ such that 
$|\alpha|+|\beta|=n$. We write this as
\[ \Irr(W_n)=\{\chi^{(\alpha,\beta)} \mid (\alpha,\beta) \vdash n\}.\]
Now let $\chi \in\Irr(W)$. We associate with $\chi$ two invariants
$d(\chi)$ and $j_0(\chi)$, as follows. Let $(\alpha,\beta) \vdash n$ be
such that $\chi=\chi^{(\alpha,\beta)}$. Choose $m \geq 0$ such that we can 
write 
\[ \alpha=(0 \leq \alpha_1 \leq \alpha_2 \leq \ldots \leq \alpha_{m+1})
\qquad \mbox{and} \qquad \beta=(0 \leq \beta_1 \leq \beta_2 \leq \ldots \leq 
\beta_m).\]
As in \cite[\S 4.5]{LuB}, we have a corresponding ``symbol''
\[ \Lambda_m(\chi):=\binom{\lambda_1, \lambda_2,\ldots,\lambda_{m+1}}{\mu_1,
\mu_2,\ldots,\mu_m} \]
where $\lambda_i:=\alpha+i-1$ for $1 \leq i \leq m+1$ and $\mu_i=\beta_i+
i-1$ for $1\leq i \leq m$. We set 
\begin{align*}
d(\chi) &:= \mbox{ number of $i \in \{1,\ldots,m\}$ such that
$\mu_i \not\in \{\lambda_1,\lambda_2,\ldots,\lambda_{m+1}\}$}, \\
j_0(\chi) &:=\sum_{1 \leq i \leq m} \min\{\alpha_{i+1},\beta_i\}.
\end{align*}
(Note that these definitions do not depend on the choice of $m$.)
By \cite[22.14]{lusztig}, we have $f_\chi =2^{d(\chi)}$; furthermore, 
\[ \mbox{$\chi$ is special} \qquad \Longleftrightarrow \qquad \lambda_i 
\leq \mu_i \leq \lambda_{i+1} \quad \mbox{for $1\leq i \leq m$}.\]
Now, by \cite[(3.2.4)]{kottwitz}, the following hold:
\begin{itemize}
\item[(a)] $\langle \rho_{l,j},\chi\rangle_{W_n}=0$ unless
$\chi$ is special and $j+l=|\beta|$.
\item[(b)] If $\chi$ is special and $j+l=|\beta|$, then 
\[ \langle \rho_{l,j},\chi\rangle_{W_n}= \binom{d(\chi)}{j_0(\chi)-j} 
\qquad \mbox{(binomial coefficient)};\]
in particular, the multiplicity is zero unless $j_0(\chi)-d(\chi)\leq 
j \leq j_0(\chi)$.
\end{itemize} 
Consequently, if $\Ib$ denotes the set of all involutions in $W_n$, then
\[\rho_{\Ib}=\sum_{l,j} \rho_{l,j}=\sum_{\chi \in \SC(W_n)}
2^{d(\chi)}\, \chi.\]
\end{exemple}

\medskip
\begin{rema} \label{remkoB} By \cite[8.1]{LuB}, we have an explicit 
combinatorial description of the cuspidal and non-cuspidal two-sided 
cells in $W_n$. Let us briefly recall the main points of this description. 
Let $\CG$ be any two-sided cell and $\chi_0 \in \Irr_{\CG}(W_n)$ the 
unique special character. Let $(\alpha, \beta) \vdash n$ be such that 
$\chi_0= \chi^{(\alpha, \beta)}$. Write 
\[ \alpha=(0\leq \alpha_1 \leq \ldots \leq \alpha_{m+1}) \qquad \mbox{and}
\qquad \beta=(0 \leq \beta_1 \leq \ldots \leq \beta_m)\]
for some $m \geq 0$. Consider the corresponding symbol 
\[ \Lambda_m(\chi_0)=\binom{\lambda_1, \lambda_2,\ldots,
\lambda_{m+1}}{\mu_1, \mu_2,\ldots,\mu_m}; \qquad \mbox{see 
Example~\ref{kobn}}.\] 
We assume that $m$ is chosen such that $0$ does not appear in both rows 
of $\Lambda_m(\chi)$. 

First of all, $\CG$ is cuspidal if and only if
$n=d^2+d$ for some $d \geq 1$ and $\Lambda_m(\chi_0)$ contains each of 
the numbers $0,1,\ldots,2m$ exactly once.

Now consider the general case. Let $t_0$ be the largest entry in 
$\Lambda_m(\chi_0)$. Then $\CG$ is strongly non-cuspidal if there is 
some $i \in \{0, 1,\ldots, t_0-1\}$ which does not appear in any of the 
two rows of $\Lambda_m (\chi_0)$. Let us now assume that this is the case. 
Then there exists a parabolic subgroup $W' \subsetneqq W_n$ and a two-sided 
cell $\CG'$ of $W'$ such that $\Jrm_{W'}^{W_n}$ establishes a bijection 
\[ \Irr_{\CG'}(W') \rightarrow \Irr_{\CG}(W_n), \qquad \chi' \mapsto
\Jrm_{W'}^{W_n}(\chi').\]
More precisely, as discussed in \cite[8.1]{LuB}, the subgroup $W'$ and the
two-sided cell $\CG'$ can be chosen as follows, where $\chi_0'\in 
\Irr_{\CG'}(W')$ is the unique special character. 
\begin{itemize}
\item[(a)] There exists some $r \in \{1,\ldots,n\}$ such that $W'=W_{n-r} 
\times H_r$ where $W_{n-r}=\langle t,s_1,\ldots, s_{n-r-1} \rangle$ 
(of type $B_{n-r}$) and $H_r=\langle s_{n-r+1},\ldots,s_{n-1}
\rangle \cong \SG_r$.
\item[(b)] We have $\chi_0'=\psi_0\boxtimes \varepsilon_r$ where $\psi_0 \in
\Irr(W_{n-r})$ is special and $\varepsilon_r$ denotes the sign character 
on $H_r$; furthermore, $\Lambda_m(\chi_0)$ is obtained by increasing the 
largest $r$ entries in the symbol $\Lambda_m(\psi_0)$ by $1$. 
\item[(c)] We have $d(\chi_0)=d(\psi_0)$, $j_0(\chi_0)=j_0(\psi_0)+
\lfloor r/2 \rfloor$ and $|\beta|=|\beta'|+\lfloor r/2 \rfloor$ where
$\psi_0$ is labelled by the pair of partitions $(\alpha',\beta')\vdash n-r$.
\end{itemize}
Only (c) requires a proof here. (Both (a) and (b) are explicitly discussed
in \cite[8.1]{LuB}.) As remarked in Example~\ref{kobn}, we have 
$f_{\chi_0}=2^{d(\chi_0)}$ and $f_{\psi_0}=2^{d(\psi_0)}$. So the first
equality follows from Remark~\ref{cusp1}. To see the second equality in (c),
consider the symbol $\Lambda_m(\chi_0)$. Since $\chi_0$ is special, the 
largest $r$ entries in $\Lambda_m(\chi_0)$ are the last $r$ terms in the 
sequence
\[\lambda_1,\quad \mu_1,\quad\lambda_2,\quad\mu_2,\quad\ldots,\quad 
\lambda_m,\quad\mu_m, \quad\lambda_{m+1}.\]
Now consider the pair of partitions $(\alpha',\beta') \vdash n-r$ such that
$\psi_0=\chi^{(\alpha',\beta')}$; we also write 
$\alpha'=(0\leq \alpha_1' \leq \ldots \leq \alpha_{m+1}')$ and 
$\beta'=(0 \leq \beta_1' \leq \ldots \leq \beta_m')$. By (b), the symbol
$\Lambda_m(\chi_0)$ is obtained by increasing the largest~$r$ entries 
in the symbol $\Lambda_m(\psi_0)$ by~$1$. Consequently, the sequence 
\[\alpha_1, \quad \beta_1,\quad \alpha_2, \quad \beta_2, \quad \ldots, 
\quad \alpha_m, \quad\beta_m, \quad\alpha_{m+1}\]
is obtained from the sequence 
\[\alpha_1', \quad\beta_1',\quad \alpha_2', \quad \beta_2', \quad \ldots, 
\quad \alpha_m', \quad\beta_m', \quad\alpha_{m+1}'\]
by increasing the last $r$ terms in the latter sequence by $1$. This
then immediately yields the statements about $j_0(\chi_0)$ and $|\beta|$. 
Thus, (c) is proved. 
\end{rema}

\medskip

\begin{theo} \label{kottB} Let $W=W_n$ be of type $B_n$, as above. Let 
$\CG$ be a two-sided cell of $W_n$ and $\chi_0 \in \Irr_{\CG}(W_n)$ be 
the unique special character. Let $\CC$ be a conjugacy class of 
involutions in $W_n$. Then
\[ \langle \rho_{\CC},[C]\rangle_{W_n}=\langle \rho_{\CC},\chi_0
\rangle_{W_n}=|\CC \cap C|\qquad \mbox{for any left cell 
$C \subseteq \CG$}.\]
Thus, Kottwitz' Conjecture~\ref{coko} holds for $W_n$.
\end{theo}

\begin{proof} The first equality is seen as follows. As already remarked
in Example~\ref{class1}, we have $\langle [C],\chi_0\rangle_{W_n}=1$ for
every left cell $C \subseteq \CG$. On the other hand, by 
Example~\ref{kobn}(a), all constituents of $\rho_{\CC}$ are special. Hence,
we have $\langle \rho_{\CC}, [C]\rangle_{W_n}=\langle \rho_{\CC},\chi_0
\rangle_{W_n}$, as required.

We now show by induction on $n$ that Kottwitz' Conjecture holds. If $n=1$, 
then $W_2 \cong \SG_2$ and the assertion holds by Example~\ref{koan}. Now 
assume that $n \geq 2$. First we consider the case where $\CG$ is strongly
non-cuspidal. Let $W'=W_{n-r} \times H_r$ and $\psi_0 \in \Irr(W_{n-r})$
be as in Remark~\ref{remkoB}, where $r \in \{1,\ldots,n\}$. Then
\[ \chi_0=\Jrm_{W'}^W(\chi_0') \qquad \mbox{where} \qquad \chi_0'=\psi_0
\boxtimes \varepsilon_r.\]
Let $\CG'$ be the two-sided cell of $W'$ such that $\chi_0'\in 
\Irr_{\CG'}(W')$. We now check that the assumptions (K1), (K2), (K3) in 
Lemma~\ref{kott0} are satisfied.

Assumption (K1) certainly holds by the identity in Example~\ref{class1},
while (K2) holds by our inductive hypothesis. Now consider (K3).

Let $\chi' \in \Irr_{\CG'}(W')$ and $\CC$ be a conjugacy class of
involutions in $W_n$ such that $\CC \cap W' \neq \varnothing$. If 
$\langle \rho_{\CC \cap W'},\chi' \rangle_{W'}=0$, then the assertion 
is obvious. Now assume that $\langle \rho_{\CC \cap W'},\chi'\rangle_{W'}
\neq 0$. Then there is a conjugacy class of involutions $\CC'$ in $W'$ 
such that 
\begin{equation*}
\CC' \subseteq \CC \cap W' \qquad \mbox{and} \qquad \langle \rho_{\CC'},
\chi' \rangle_{W'}\neq 0.\tag{$\triangle$}
\end{equation*}
(We shall see that $\CC'$ is uniquely determined with property.) Since 
$W'$ is a direct product, we can write $\CC'$ as a direct product of a
conjugacy class in $W_{n-r}$ and a conjugacy class in $H_r$. Thus, using 
the notation in Examples~\ref{koan} and \ref{kobn}, we have 
\[\CC'=\CC_{l,j'} \times \CC_k \qquad \mbox{where} \qquad l,j',k\geq 0, 
\quad l+2j'\leq n-r, \quad 2k\leq r;\]
here, the class $\CC_{l',j'}\subseteq W_{n-r}$ has a representative 
$\sigma_{l,j'}$ given by the expression in Example~\ref{kobn} and 
the class $\CC_k \subseteq H_r$ has a representative $\sigma_k$ as in 
Example~\ref{koan}. (Explicitly, we have $\sigma_k=s_{n-r+1}s_{n-r+3}
\cdots s_{n-r+2k-1}$.) We note that $\sigma_{l,j'} \times \sigma_k \in 
\CC'$ is the longest element in a parabolic subgroup of $W_n$ of type $B_l
\times A_1 \times \ldots \times A_1$, where the $A_1$ factor is repeated
$j'+k$ times. Hence, since $\CC' \subseteq \CC$, we must have
\[ \CC=\CC_{l,j} \qquad \mbox{where} \qquad j=j'+k.\] 
Now, we can also write $\chi'=\psi \boxtimes \varepsilon_r$ where $\psi
\in \Irr(W_{n-r})$. Then  we obtain
\[ \langle \rho_{\CC'},\chi'\rangle_{W'}=\langle \rho_{l,j'},\psi
\rangle_{W_{n-r}} \langle \rho_{k},\varepsilon_r\rangle_{H_r},\]
Since this is assumed to be non-zero, we conclude that 
\[ \langle \rho_{l,j'},\psi \rangle_{W_{n-r}} \neq 0 \qquad \mbox{and}
\qquad \langle \rho_{k},\varepsilon_r\rangle_{H_r} \neq 0.\]
By Example~\ref{kobn}, the first condition implies that $\psi$ is special 
and, hence, $\chi'$ is special. Thus, we must have $\chi'=\chi_0'$ and 
$\psi=\psi_0$. By Example~\ref{koan}(b), the second condition implies that 
$\langle \rho_k,\varepsilon_r\rangle_{H_r}=1$ and $k=\lfloor r/2 \rfloor$. 
In particular, the class $\CC'$ in ($\triangle$) is uniquely determined. 
Combining these statements, we obtain that 
\[\langle \rho_{\CC \cap W'},\chi'\rangle_{W'}=\langle \rho_{\CC'},\chi'
\rangle_{W_n}=\langle \rho_{l,j'},\psi_0 \rangle_{W_{n-r}}.\]
Since $\chi'=\chi_0'$, we have $\chi_0=\Jrm_{W'}^{W_n}(\chi')$; since $\CC=
\CC_{l,j}$, we are finally reduced to showing that 
\[ \langle \rho_{l,j'},\psi_0\rangle_{W_{n-r}}\leq \langle \rho_{l,j}, 
\chi_0\rangle_{W_n} \qquad \mbox{where} \qquad j=j'+k \quad \mbox{and} 
\quad k=\lfloor r/2 \rfloor.\]
But, by Remark~\ref{remkoB}(c), we have $d(\chi_0)=d(\psi_0)$ and
$j_0(\chi_0)=j_0(\psi_0)+ \lfloor r/2 \rfloor$. Hence, the multiplicity 
formula in Example~\ref{kobn} shows that we actually have 
\[\langle \rho_{l,j'},\psi_0\rangle_{W_{n-r}}=\langle \rho_{l,j}, 
\chi_0\rangle_{W_n}.\] 
Thus, (K3) is satisfied and so Kottwitz' Conjecture holds for $\CG$. Since 
the longest element $w_0 \in W_n$ is central in $W_n$, we can apply 
Lemma~\ref{epsrho1} which shows that Kottwitz's Conjecture will
also hold for $\CG w_0$. By \cite[8.1]{LuB}, these arguments cover all
non-cuspidal two-sided cells in $W_n$.

It remains to consider the case where $\CG$ is a cuspidal two-sided cell.
By Remark~\ref{remkoB}, such a two-sided cell can only exist if $n=d^2+d$
for some $d \geq 1$, in which case it is uniquely determined. So let
us now assume that $n=d^2+d$ where $d \geq 1$. Let $W_n=\coprod_{0 \leq
i\leq N}\CG_i$ be the partition into two-sided cells where $\CG_0$ is
the unique cuspidal two-sided cell. For $0 \leq i \leq N$, let $\chi_i\in 
\Irr_{\CG_i}(W_n)$ be the unique special character and let $C_i\subseteq
\CG_i$ be a left cell. Let $\CC$ be a conjugacy class of involutions. To
obtain a statement about $\langle \rho_{\CC},\chi_0\rangle_{W_n}$, we
consider 
\[ \sum_{0 \leq i \leq N} \chi_i(1)\langle \rho_{\CC},\chi_i\rangle_{W_n}=
\Big\langle \rho_{\CC}, \sum_{0 \leq i \leq N} \chi_i(1)\chi_i
\Big\rangle_{W_n}.\]
Since all constituents of $\rho_{\CC}$ are special, the sum on
the right hand side can be extended over all $\chi \in \Irr(W_n)$, 
in which case we just obtain the character of the regular representation
of $W_n$. Hence, the right hand side equals $\rho_{\CC}(1)$. Now, for any
$i \geq 1$, we already know that Kottwitz's Conjecture holds for $\CG_i$ 
and so $\langle \rho_{\CC},\chi_i\rangle_{W_n}=\langle \rho_{\CC},[C_i]
\rangle_{W_n}=|\CC\cap C_i|$. Hence, we find that 
\[ \chi_0(1)\langle \rho_{\CC},\chi_0\rangle_{W_n}=\rho_{\CC}(1)-
\sum_{1 \leq i \leq N} \chi_i(1)|\CC \cap C_i|.\] 
On the other hand, using the identity in Example~\ref{class1}, we obtain
\[ \sum_{0 \leq i \leq N} \chi_i(1)|\CC \cap C_i|=
\sum_{0 \leq i \leq N} |\CC \cap \CG_i|=|\CC|.\]
Hence, we find that 
\[ \chi_0(1)|\CC\cap C_0|=|\CC|-\sum_{1\leq i \leq N} \chi_i(1)
|\CC\cap C_i|.\]
Since $\rho_{\CC}(1)=|\CC|$, we deduce that 
\[ \chi_0(1)\langle \rho_{\CC}, \chi_0\rangle_{W_n}=\chi_0(1)
|\CC\cap C_0|\]
and so $\langle \rho_{\CC},[C_0]\rangle_{W_n}=\langle \rho_{\CC},
\chi_0\rangle_{W_n} =|\CC \cap C_0|$, as required.
\end{proof}

\section{Kottwitz' Conjecture for type $D_n$}\label{section:kottDn}

Throughout this section, let $n \geq 2$ and $W=W_n'$ be of type $D_n$, 
with generators $u,s_1,\ldots,s_{n-1}$ and diagram given as follows.
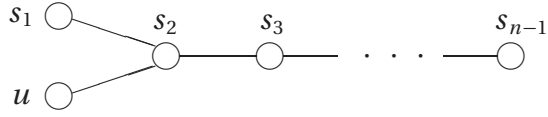
\begin{center}
\begin{picture}(220,40)
\put( 41,  3){\circle{10}}
\put( 41, 33){\circle{10}}
\put( 46,  4){\line(3,1){31}}
\put( 46, 32){\line(3,-1){31}}
\put( 81, 18){\circle{10}}
\put( 86, 18){\line(1,0){29}}
\put(120, 18){\circle{10}}
\put(125, 18){\line(1,0){20}}
\put(155, 15){$\cdot$}
\put(165, 15){$\cdot$}
\put(175, 15){$\cdot$}
\put(185, 18){\line(1,0){20}}
\put(210, 18){\circle{10}}
\put( 22, 31){$s_1$}
\put( 23,  0){$u$}
\put( 76, 28){$s_2$}
\put(116, 28){$s_3$}
\put(204, 28){$s_{n{-}1}$}
\end{picture}
\end{center}
By convention, we will also set $W_0'=W_1'=\{1\}$. The aim of this section 
is to prove that Conjecture~\ref{coko} holds for $W_n'$ (where we also rely
on some results in \cite{tkott}). For this purpose, it will be convenient 
to use an embedding of $W_n'$ into the group $W_n$ of type $B_n$, with 
generators $t,s_1,\ldots,s_{n-1}$ and diagram as in the previous section. 
Setting $u=ts_1t$ (and identifying the remaining generators $s_1,\ldots,
s_{n-1}$), we can identify $W_n'$ with a subgroup of $W_n$. Thus, we have 
$W_n\cong W_n' \rtimes \langle \theta\rangle$ where $\theta \colon W_n'
\rightarrow W_n'$ is the automorphism given by conjugation with~$t$. In 
this setting, a large part of the argument will be analogous to that for 
type $B_n$. However, when $n$ is even, there are some particularly 
intricate questions to solve concerning the unique conjugacy class of 
involutions in $W_n'$ which is not invariant under~$\theta$.

\medskip

\begin{exemple} \label{kodn} Let $\CC'$ be a conjugacy class of 
involutions in $W_n'$. If $\theta(\CC')= \CC'$, then $\CC'$ is a conjugacy 
class in $W_n$ and the decomposition of $\rho_{\CC'}$ into irreducible 
characters of $W_n'$ is given by formulae similar to those for type 
$W_n$ in Example~\ref{kobn}; see \cite[\S 3.3]{kottwitz}. In particular, 
we have
\begin{equation*}
\langle \rho_{\CC'},\chi\rangle_{W_n'}=0 \qquad \mbox{unless
$\chi \in \Irr(W_n')$ is special and can be extended to $W_n$}. \tag{a}
\end{equation*}
Classes which are not $\theta$-invariant can only exist if $n$ is even,
and then we will also encounter characters which can not be extended
to $W_n$. So let us now assume that $n$ is even. Let $\CC_0'$ be the 
conjugacy class of $W_n'$ containing the element
\[ \sigma_{0,n/2}:=s_1s_3s_5 \cdots s_{n-1}.\]
Then $\theta(\CC_0')\neq \CC_0'$ and $\{\CC_0',\theta(\CC_0')\}$ is the
only pair of conjugacy classes of involutions with this property; see
\cite[3.4.12]{gepf}. To describe the decomposition of $\rho_{\CC_0'}$ into
irreducible characters, we introduce some further notation. For every
partition $\alpha \vdash n/2$, we define two characters $\chi^{\alpha,
\pm 1} \in \Irr(W_n')$, as follows. Let $H_n=\langle s_1, \ldots,s_{n-1}
\rangle \cong \SG_n$. Let $2\alpha^*$ denote the partition of $2n$ obtained
by multiplying all parts of the conjugate partition $\alpha^*$ by $2$ and 
consider the corresponding Young subgroup $H_{2\alpha^*} \subseteq H_n$. 
(We have $H_{2\alpha^*} \cong \SG_{2\alpha^*}$.). Let 
$\varepsilon_{2\alpha^*}$ be the sign character of $H_{2\alpha^*}$ and let 
$w_{2\alpha^*}$ be the longest element in $H_{2\alpha^*}$. Then,
by \cite[5.3.2]{gepf}, there is a unique $\chi^{\alpha,\pm 1}\in \Irr(W_n')$
such that $\bb_\chi=\ell(w_{2\alpha^*})$ and 
\[\Ind_{H_{2\alpha^*}}^{W_n'}(\varepsilon_{2\alpha^*})=\chi^{\alpha,+1}+
\mbox{ sum of various $\chi \in \Irr(W_n')$ with $\bb_\chi>
\ell(w_{2\alpha^*})$};\]
furthermore, $\chi^{\alpha,-1}$ is defined as the conjugate of 
$\chi^{\alpha,+}$ under $\theta$. It is well-known that 
$\{\chi^{\alpha,\pm 1} \mid \alpha \vdash n/2\}$ are precisely the 
irreducible characters of $W_n'$ which can not be extended to $W_n$; see 
\cite[4.6]{LuB}, \cite[\S 5.6]{gepf}. By \cite[\S 3.3]{kottwitz}, the 
decompositions of $\rho_{\CC_0'}$ and $\rho_{\theta(\CC_0')}$ into 
irreducible characters are given as follows: 
\begin{equation*}
\rho_{\CC_0'}=\sum_{\alpha \vdash n/2} \chi^{\alpha,\nu_\alpha} \qquad
\mbox{and} \qquad \rho_{\theta(\CC_0')}=\sum_{\alpha \vdash n/2} 
\chi^{\alpha,-\nu_\alpha} \tag{b}
\end{equation*}
where $\nu_\alpha\in \{\pm 1\}$ for all $\alpha \vdash n/2$.
\end{exemple}

\medskip
Note that the above signs have not been determined in \cite{kottwitz}. 
It will be essential to fix these signs in order to prove Kottwitz'
Conjecture. In fact, the following example shows that this conjecture 
can only hold if $\nu_\alpha=+1$ for all $\alpha \vdash n/2$.

\medskip

\begin{exemple} \label{kodn1} Assume that $n$ is even and ley $\CC_0'$ be
the conjugacy class of the element $\sigma_{0,n/2}$, as above. By 
\cite[4.6.10]{LuB}, we have for any $\chi=\chi^{\alpha,\pm 1}$ where 
$\alpha \vdash n/2$:
\begin{equation*}
f_\chi=1 \qquad \mbox{and} \qquad \ab_{\chi}=\bb_{\chi}=\ell(w_{2\alpha^*})
\tag{a} \end{equation*}
Consequently, each $\chi^{\alpha,\pm 1}$ is special and we have 
\begin{equation*}
\chi^{\alpha,+1}=\Jrm_{H_{2\alpha^*}}^{W_n'}(\varepsilon_{2\alpha^*})
\qquad \mbox{for any $\alpha \vdash n/2$};\tag{b}
\end{equation*}
see \cite[4.6.2]{LuB}. Let $\CG_\alpha^{\pm}$ denote the two-sided cell 
such that $\chi^{\alpha, \pm 1} \in \Irr_{\CG_\alpha^{\pm}}(W_n')$. Then
$\CG_\alpha^{\pm}$ is smooth, by (a) and Lemma~\ref{lem:smooth}. We also 
have:
\begin{equation*}
\CC_0' \cap \CG_{\alpha}^{+} \neq \varnothing \qquad \mbox{and} \qquad
\theta(\CC_0') \cap \CG_{\alpha}^{-} \neq \varnothing \qquad \mbox{for all
$\alpha \vdash n/2$}.\tag{c}
\end{equation*}
Indeed, (a), (b) and Example~\ref{lem:useful} show that $w_{2\alpha^*}\in 
\CG_\alpha^{+}$. Now recall that $H_{2\alpha^*}$ is isomorphic to a direct 
product of various symmetric groups of even degrees, where the sum of all
these degrees is $n$. Since the longest element in $\SG_{2m}$ (any $m 
\geq 1$) is a product of $m$ disjoint $2$-cylces, we see that 
$w_{2\alpha^*}$ is a product of $n/2$ disjoint $2$-cycles and so we have 
$w_{2\alpha^*} \in \CC_0'$. Hence, $\CC_0'$ is the unique conjugacy 
class of involutions in $W_n'$ such that $\CC_0' \cap \CG_{\alpha}^{+} \neq 
\varnothing$ (see Corollary~\ref{coro:conjugues}). Similarly, 
$\theta(\CC_0')$ is the unique conjugacy class of involutions in $W_n'$ 
such that $\theta(\CC_0') \cap \CG_{\alpha}^{+} \neq \varnothing$.

In particular, if $C$ is a left cell contained in $\CG_\alpha^+$, then 
$[C]=\chi^{\alpha,+1}$ and $|\CC_0'\cap C|=1$. So, if Kottwitz' 
Conjecture holds for $W_n'$, then we must have $\langle \rho_{\CC_0'},
\chi^{\alpha,+1}\rangle=1$.
\end{exemple}

\medskip

Now, determining the signs in Example~\ref{kodn}(b) is related to the 
subtle issue of distinguishing the two characters $\chi^{\alpha,+1}$ and 
$\chi^{\alpha,-1}$ for a given partition $\alpha \vdash n/2$. We shall
need the following version of the ``branching rule'' for the characters
of $W_n'$.

\medskip

\begin{lem} \label{solveD2} Assume that 
$n \geq 2$ is even. Consider the parabolic subgroup $W'=W_{n-2}'\times H_2$ 
where $W_{n-2}'=\langle u,s_1,\ldots,s_{n-3}\rangle$ (type $D_{n-2}$) and 
$H_2=\langle s_{n-1} \rangle$. Let $\alpha' \vdash (n-2)/2$ and denote by 
$\varepsilon_1$ the sign character on the factor $H_2$. Then
\[ \Ind_{W'}^{W_n'}\bigl(\chi^{\alpha',+1} \boxtimes \varepsilon_1\bigr)=
\sum_{\alpha} \chi^{\alpha,+1} \quad+\quad  \mbox{``further terms''},\]
where the sum runs over all partitions $\alpha \vdash n/2$ such that
$\alpha$ is obtained by increasing one part of $\alpha'$ by~$1$; the
expression ``further terms'' stands for a sum of various $\chi \in
\Irr(W_n')$ which can be extended to $W_n$. In particular,
\[ \Big\langle \Ind_{W'}^{W_n'}\bigl(\chi^{\alpha',+1}\boxtimes
\varepsilon_1\bigr),\chi^{\alpha,-1} \Big\rangle_{W_n'}=0 \qquad 
\mbox{for all $\alpha \vdash n/2$}.\]
\end{lem}

A proof can be found in \cite[\S 3]{tkott}. 

\medskip

\begin{prop} \label{signD} Assume that $n\geq 2$ is even. Then, with the
notation in Example~\ref{kodn}, we have 
\[\rho_{\CC_0'}=\sum_{\alpha \vdash n/2} \chi^{\alpha,+1} \qquad \mbox{and}
\qquad \rho_{\theta(\CC_0')}=\sum_{\alpha \vdash n/2} \chi^{\alpha,-1}.\]
\end{prop}

\begin{proof} We prove this by induction on $n/2$. If $n=2$, then the
assertion is easily checked directly. The character table of $W_2'=\langle
u,s_1\rangle$ with the appropriate labelling of the characters is given
as follows.
\[\begin{array}{crrrr} \hline & 1 & s_1 & u & s_1u \\ \hline
\chi^{(2,\varnothing)} & 1 & 1 & 1 & 1 \\
\chi^{(11,\varnothing)} & 1 & -1 & -1 & 1 \\
\chi^{(1,+)} & 1 & -1 & 1 & -1  \\
\chi^{(1,-)} & 1 & 1 & -1 & -1 \\ \hline
\end{array} \qquad\quad \begin{array}{c} \rho_{\CC_0'}=\chi^{(1,+)},
\\ \\ \rho_{\theta(\CC_0')}=\chi^{(1,-)}. \\ \end{array}\]
Now assume that $n\geq 4$. Let
$W'=W_{n-2}'\otimes H_2$ be as in Lemma~\ref{solveD2}. As already noted 
in the above proof, the intersection $\CC_0' \cap W'$ is just the
conjugacy class of $W'$ containing $\sigma_{0,n/2}$. Hence, we are in
the setting of Lemma~\ref{indko1} and so 
\[ \big\langle \rho_{\CC_0'},\chi^{\alpha,-}\big\rangle_{W_n'}\leq 
\Big\langle \Ind_{W'}^{W_n'}(\rho_{\CC_0'\cap W'}),\chi^{\alpha,-}
\Big\rangle_{W_n'} \qquad \mbox{for all $\alpha \vdash n/2$}.\]
It will now be sufficient to show that the scalar product on the right
hand side is zero for all $\alpha \vdash n/2$. Now, since 
$\sigma_{0,n/2}=\sigma_{0, (n-2)/2}\times s_{n-1} \in W_{n-2}'\times H_2$,
we can apply induction and obtain 
\[ \rho_{\CC_0'\cap W'}=\Bigl(\sum_{\alpha' \vdash n/2} \chi^{\alpha',+1}
\Bigr) \boxtimes \varepsilon_1.\]
Then Lemma~\ref{solveD2} implies that 
\[ \Big\langle \Ind_{W'}^{W_n'}\bigl(\rho_{\CC_0'\cap W'}\bigr),
\chi^{\alpha,-1}\Big\rangle_{W_n'}=0 \qquad \mbox{for all $\alpha \vdash 
n/2$},\]
as required.
\end{proof}

\medskip

\begin{rema} \label{preD} Let $\CC$ be any conjugacy classes of involutions 
in $W_n$. We can associate with $\CC$ a character $\tilde{\rho}_{\CC}$ of 
$W_n$, as follows. If $\CC$ is contained in $W_n'$, let $\rho_{\CC}$ be the 
character of $W_n'$ as defined in Definition~\ref{defko}. Then 
$\tilde{\rho}_{\CC}$ will be the canonical extension described in 
\cite[\S 2]{gema}. If $\CC$ is contained in the coset $W_n't$, then we 
consider a similar extension of the ``twisted'' character defined in 
\cite[4.2]{kottwitz}. Then one can show that
\[ \big\langle \tilde{\rho}_{\CC},\Ind_{W_n'}^{W_n}([C])\big\rangle_{W_n}= 
|\CC \cap (C \cup tC)| \qquad \mbox{for any left cell 
$C \subseteq W_n'$}. \]
The proof of this equality, although quite similar to that of
Theorem~\ref{kottB}, requires a number of preparations concerning
``twisted'' involutions with respect to the non-trivial graph automorphism
of $W_n'$. Furthermore, the sets $C \cup tC$ can actually be interpreted 
as left cells for $W_n$, but with respect to the non-constant weight 
function with value $0$ on $t$ and value $1$ on all $s_i$; the characters 
of the corresponding left cell modules of $W_n$ are given by the induced 
characters on the left hand side. The whole argument is worked out in
\cite[\S 5]{tkott}.
\end{rema}
%The appropriate framework for doing this is provided by
%\cite[\S 2.4]{geja}; 

%By a combination of the results in \cite[\S 3.3, 
%\S 5.4]{kottwitz} with \cite[Prop.~3.5]{gema}, we obtain an explicit 
%description of the decomposition of $\tilde{\rho}_{\CC}$ into irreducible 
%characters of $W_n$, similar to that in Example~\ref{kobn}. 
%

\begin{coro} \label{kottD} Let $W=W_n'$ be of type $D_n$, as above. Let
$\CG$ be a two-sided cell of $W_n'$ and $\chi_0 \in \Irr_{\CG}(W_n')$ be 
the unique special character. Let $\CC'$ be a conjugacy class of 
involutions in $W_n'$. Then
\[ \langle \rho_{\CC},[C]\rangle_{W_n'}=\langle \rho_{\CC},\chi_0
\rangle_{W_n'}=|\CC \cap C|\qquad \mbox{for any left cell 
$C \subseteq \CG$}.\]
Thus, Kottwitz' Conjecture~\ref{coko} holds for $W_n'$. In particular,
if $n$ is even and $\CC_0'$ denotes the conjugacy class of the element
$\sigma_{0,n/2}=s_1s_3\cdots s_{n-1}$, then 
\[ \langle \rho_{\CC_0'},[C]\rangle_{W_n'}=|\CC_0'\cap C|=1
\qquad \mbox{for any left cell $C\subseteq \CG_\alpha^+$ and
any $\alpha \vdash n/2$},\]
where $\CG_\alpha^+$ is the smooth two-sided cell as in Example~\ref{kodn1}.
\end{coro}

\begin{proof} The equality $\langle \rho_{\CC},[C]\rangle_{W_n'}=
\langle \rho_{\CC},\chi_0\rangle_{W_n'}$ is shown as in the proof
of Theorem~\ref{kottB}, using Example~\ref{kodn}(a), (b). 

To prove Kottwitz' Conjecture, let us first deal with the case where $\CC'$ 
is a conjugacy class of involutions in $W_n'$ such that $\theta(\CC')=
\CC'$. Then $\CC'$ is a conjugacy class in $W_n$ and we can use the identity
in Remark~\ref{preD}. By Frobenius reciprocity we obtain:
\begin{equation*}
\langle \rho_{\CC'},[C]\rangle_{W_n'}= |\CC' \cap C| \qquad 
\mbox{for any left cell $C \subseteq W_n'$}.\tag{a}
\end{equation*}
It now remains to deal with the case where $n$ is even and $\CC'$ is
such that $\theta(\CC')\neq \CC'$. Let $\CC'=\CC_0'$ be the 
conjugacy class in Example~\ref{kodn}. First we will show that 
\begin{equation*}
\langle \rho_{\CC_0'},[C]\rangle_{W_n'}\leq |\CC_0'\cap C| \qquad 
\mbox{for any left cell $C\subseteq W_n'$}.\tag{b}
\end{equation*}
Indeed, let $C$ be a left cell in $W_n'$. If $\langle \rho_{\CC_0'},
[C]\rangle_{W_n'}=0$, then (a) is obvious. Now assume that $\langle 
\rho_{\CC_0'}, [C]\rangle_{W_n'}\neq 0$. Then, by Proposition~\ref{signD}, 
there exists some $\alpha \vdash n/2$ such that $\langle \chi^{\alpha,+1},
[C]\rangle_{W_n'}\neq 0$. So, using the notation in Example~\ref{kodn1}, 
we have $C \subseteq \CG_\alpha^+$ and $\CC_0'\cap \CG_\alpha^+\neq 
\varnothing$. Since $\CG_\alpha^+$ is smooth, we have $[C]=
\chi^{\alpha,+1}$ (see Lemma~\ref{lem:smooth}) and $\CC_0' \cap C \neq 
\varnothing$ (see Corollary~\ref{coro:conjugues}). Consequently, we see
that (b) holds. Once this is established, it actually follows that we must 
have equality in (b). Indeed, let $W=\coprod_{1\leq i \leq m} C_i$ be the
partition into left cells. By (b), we have 
\[ \Big\langle \rho_{\CC_0'},\sum_{1\leq i \leq m} [C_i]\Big\rangle_{W_n'}=
\sum_{1\leq i \leq m} \langle \rho_{\CC_0'},[C_i]\rangle_{W_n'}
\leq \sum_{1l\leq i \leq m} |\CC_0'\cap C_i|=|\CC_0'|.\]
But, $\sum_{1\leq i \leq m} [C_i]$ is the character of the regular
representation of $W_n'$ and so the left hand side also equals $|\CC_0'|=
\rho_{\CC_0'}(1)$. So all the inequalities in (b) must be equalities,
as required. The argument for $\theta(\CC_0')$ is completely analogous. 
Note that, by Example~\ref{kodn}(b), the character $\rho_{\theta(\CC_0')}$ 
is the conjugate of $\rho_{\CC_0'}$ under $\theta$. 
\end{proof}

%let $C_\alpha'$ be the left cell of $W_n'$ which contains
%the element $w_\alpha'$ in Example~\ref{kodn}. Then 
%\[ \chi^{\alpha,+}=[C_\alpha']; \qquad \mbox{see ??}.\]
%be the left cell of $W_n'$ which contains
%
%

\end{document}